\documentclass{amsart} 
\usepackage{etex}
\usepackage[utf8]{inputenc}
\usepackage{graphicx}
\usepackage{amsmath}
\usepackage{amssymb}
\usepackage{lmodern}
\usepackage{fontenc}[T1]
\usepackage{enumerate}
\usepackage{enumitem}
\usepackage[english]{babel}
\usepackage{ifthen}
\usepackage{oubraces}
\usepackage{nicefrac}
\usepackage{mathdots}

\usepackage{multirow}
\usepackage{xspace}
\usepackage{verbatim}

\usepackage{pdfsync}

\DeclareMathOperator{\End}{End}

\usepackage{tikz}
\usetikzlibrary{shapes}
\usetikzlibrary{arrows,decorations.markings}
\usepackage{pgf}
\usetikzlibrary{automata}
\usetikzlibrary{chains,decorations.pathmorphing,positioning,fit}
\usetikzlibrary{decorations.shapes,calc,backgrounds}
\usetikzlibrary{decorations.text,matrix}
\usetikzlibrary{arrows,shapes.geometric,shapes.symbols,scopes}
\usetikzlibrary{decorations.markings}

\definecolor{darkgreen}{rgb}{0.0,0.7,0.0}

\newenvironment{AUTH}{\noindent\color{darkgreen} Authors:}{}

\newenvironment{LC}{\noindent\color{green} LC:}{}

\newenvironment{vd}{\noindent\color{blue} \colorbox{blue}{\color{black} VD:}}{}

\newenvironment{ME}{\noindent\color{magenta} ME:} {}

\newenvironment{ml}{\noindent\color{red} \colorbox{blue}{\color{black} Referee:}}{}

\newtheorem{theorem}{Theorem}
\newtheorem{proposition}[theorem]{Proposition}
\newtheorem{lemma}[theorem]{Lemma}
\newtheorem{corollary}[theorem]{Corollary}

\theoremstyle{definition}
\newtheorem{definition}[theorem]{Definition}
\newtheorem{remark}[theorem]{Remark}

\newtheorem{example}[theorem]{Example}

\usepackage{prettyref}
\newcommand{\prref}[1]{\prettyref{#1}}
\newif{\ifshort}\shorttrue
\shortfalse
\ifshort
\newrefformat{alg}{Alg.~\ref{#1}}
 \newrefformat{thm}{Thm.~\ref{#1}}
 \newrefformat{lem}{Lem.~\ref{#1}}
 \newrefformat{def}{Def.~\ref{#1}}
 \newrefformat{cor}{Cor.~\ref{#1}} 
 \newrefformat{prop}{Prop.~\ref{#1}}
 \newrefformat{sec}{Sect.~\ref{#1}}
 \newrefformat{kap}{Chap.~\ref{#1}}
 \newrefformat{ex}{Ex.~\ref{#1}}
 \newrefformat{app}{App.~\ref{#1}}
 \newrefformat{alg}{Alg.~\ref{#1}}
 \newrefformat{eq}{(\ref{#1})}%
 \newrefformat{tab}{Tab.~\ref{#1}}
 \newrefformat{fig}{Fig.~\ref{#1}}
 \newrefformat{rem}{Rem.~\ref{#1}}
 \newrefformat{ap}{{\sc Appendix}}

\else
\newrefformat{alg}{Algorithm~\ref{#1}}
\newrefformat{thm}{Theorem~\ref{#1}}
\newrefformat{lem}{Lemma~\ref{#1}}
\newrefformat{def}{Definition~\ref{#1}}
\newrefformat{cor}{Corollary~\ref{#1}}
\newrefformat{prop}{Proposition~\ref{#1}}
\newrefformat{sec}{Section~\ref{#1}}
\newrefformat{subsec}{Subsection~\ref{#1}}
\newrefformat{kap}{Chapter~\ref{#1}}
\newrefformat{ex}{Example~\ref{#1}}
\newrefformat{rem}{Remark~\ref{#1}}
\newrefformat{fig}{Figure~\ref{#1}}
\newrefformat{app}{Appendix~\ref{#1}}
\newrefformat{eq}{Equation~(\ref{#1})}
\newrefformat{tab}{Table~\ref{#1}}
\newrefformat{ap}{{\sc Appendix}}
\fi

\newcommand{\ispath}[2]{\text{\sc{Ispath}}(#1,#2)}

\newcommand{\subst}{substitution\xspace}

\newcommand{\exe}{extended equation\xspace}

\newcommand{\solu}{solution\xspace}
\newcommand{\invol}{involution\xspace}

\newcommand{\edtol}{EDT0L\xspace} 
\newcommand{\tra}{transition\xspace}

\newcommand{\IFF}{if and only if\xspace}
\renewcommand{\hom}{homomorphism\xspace}
\newcommand{\homs}{homomorphisms\xspace}
\newcommand{\Endo}{endomorphism\xspace}
\newcommand{\Endos}{endomorphisms\xspace}

\newcommand{\morph}{morphism\xspace}

\newcommand{\lds}{, \ldots ,}

\newcommand{\ra}{\longrightarrow}

\usepackage[framemethod=tikz]{mdframed} 

\newcommand{\Apos}{A_{+}}

\newcommand{\Apone}{A_{\pm}}

\newcommand{\FG}[1]{\text{F}({#1})}
\newcommand{\FGA}{\FG \Apos}

\newcommand{\muinit}{\mu_{\mathrm{init}}}
\newcommand{\Winit}{W_{\mathrm{init}}}
\newcommand{\Xinit}{\cX_{\mathrm{init}}}

\newcommand{\arc}[1]{\overset{#1}\ra}

\newcommand{\set}[2]{\left\{#1\mathrel{\left|\vphantom{#1}\vphantom{#2}\right.}#2\right\}}

\newcommand{\oneset}[1]{\left\{\mathinner{#1}\right\}}
\newcommand{\os}{\oneset}
\newcommand{\sm}{\setminus}
\newcommand{\es}{\emptyset}
\newcommand{\sse}{\subseteq}

\newcommand{\abs}[1]{\left|\mathinner{#1}\right|}
\newcommand{\Abs}[1]{\left\Vert\mathinner{#1}\right\Vert}
 
\newcommand{\gen}[1]{\left< \mathinner{#1} \right>}

\newcommand{\N}{\ensuremath{\mathbb{N}}}
\newcommand{\Z}{\ensuremath{\mathbb{Z}}}

\newcommand{\F}{\ensuremath{\mathbb{F}}}

\newcommand{\M}{\ensuremath{\mathbb{M}}}
\newcommand{\MMA}{\ensuremath{\mathbb{M}(A)}}

\newcommand{\NP}{\ensuremath{\mathsf{NP}}}
\renewcommand{\P}{\ensuremath{\mathsf{P}}}
\newcommand{\NSPACE}{\ensuremath{\mathsf{NSPACE}}}
\newcommand{\DSPACE}{\ensuremath{\mathsf{DSPACE}}}

\newcommand{\DTIME}{\ensuremath{\mathsf{DTIME}}}

\renewcommand{\phi}{\varphi}
\newcommand{\eps}{\varepsilon}

\newcommand{\oo}{\omega}
\newcommand{\alp}{\alpha}
\newcommand{\bet}{\beta}

\newcommand{\lam}{\lambda}
\newcommand{\sig}{\sigma}

\newcommand{\Sig}{\Sigma}
\newcommand{\Gam}{\GG}
\newcommand{\Del}{\Delta}

\newcommand\GG{\Gamma}

\newcommand\Lam{\Lambda}
\newcommand\OO{\Omega}

\newcommand{\Oh}{\mathcal{O}}

\newcommand{\id}[1]{\mathrm{id}_{#1}}
\newcommand{\cA}{\mathcal{A}}

\newcommand{\cC}{\mathcal{C}}
\newcommand{\cD}{\mathcal{D}}

\newcommand{\cF}{\mathcal{F}}

\newcommand{\cL}{\mathcal{L}}

\newcommand{\cP}{\mathcal{P}}

\newcommand{\cX}{\mathcal{X}}

\newcommand{\cZ}{\mathcal{Z}}
\newcommand{\cR}{\mathcal{R}}

\newcommand{\cSol}{\mathrm{Sol}}

\newcommand{\ov}[1]{\overline{#1}}

\newcommand{\oi}[1]{{#1}^{-1}}

\newcommand{\E}[1]{\mathop{\mathrm E}\left[\, #1 \,\right]}

\newcommand{\bcc}{96n+ 6\abs{\Winit}} 
\newcommand{\pcc}{104n+ 6\abs{\Winit}} 

\begin{document}

\title{Solution sets for
equations over free groups  are {EDT0L} languages}

\author{Laura Ciobanu}

\address{Institut de math\'ematiques, Universit{\'e} de Neuch{\^a}tel,
Rue Emile-Argand 11,  CH-2000 Neuch{\^a}tel, Switzerland}
\email{laura.ciobanu@unine.ch}

\author{Volker Diekert}

\address{Institut f\"ur Formale Methoden der Informatik,  Universit\"at Stuttgart,
  Universit\"atsstr. 38, D-70569 Stuttgart, Germany}
  \email{diekert@fmi.uni-stuttgart.de}

\author{Murray Elder}
\address{School of Mathematical and Physical Sciences,
The University of Newcastle,
Callaghan NSW 2308, Australia}
\email{Murray.Elder@newcastle.edu.au}



\thanks{
Research supported by the  Australian Research Council  (Future Fellowship  FT110100178), the Swiss National Science 
Foundation (Professorship FN PP00P2-144681/1), and a Universit{\'e} de  Neuch\^atel Overhead Grant.}

\begin{abstract}
We show that, given an equation over a finitely generated free group, the set of all solutions in reduced words forms an effectively constructible EDT0L language. In particular, the set of all solutions in reduced words is an indexed language in the sense of Aho. The language characterization we give, as well as further questions about the existence or finiteness of solutions, follow from our explicit construction of a finite directed graph which encodes all the solutions. Our result incorporates the recently invented recompression technique of Je\.z, and a new way to integrate solutions of linear Diophantine equations into the process.

As a byproduct of our techniques, we improve the complexity from quadratic nondeterministic space in previous works to ${\ensuremath{\mathsf{NSPACE}}}(n\log n)$ here.  
\end{abstract}

\maketitle

\subsection*{Introduction}
In this paper we prove that the set of all solutions, as reduced words, to an equation in a finitely generated free group or free monoid with involution, has a description as an EDT0L language. Furthermore, we show that this description can be computed in 
 $\NSPACE(n\log n)$, where $n$ is the length of the equation plus  
 the number of generators of the group or monoid.

We construct a finite graph, of singly exponential size $2^{\Oh(n\log n)}$, with nodes labeled by equations of bounded size 
plus some additional data, and directed edges corresponding to  transformations applied to the equations. More precisely, the edges are labeled by endomorphisms of a free monoid $C^*$, where $C$ is a finite alphabet which includes the  group or monoid generators. The graph, viewed as a nondeterministic finite automaton, produces a rational language of \Endos of $C^*$. We show that the set of all such endomorphisms applied to a particular `seed' word gives the full set of solutions to the input equation as reduced words. Thus, by the definition of Asveld \cite{Asveld1977}, we obtain that the solution set is an EDT0L language, and therefore an indexed language. 
Moreover, one can decide if there are zero, infinitely or finitely many solutions simply by checking if the graph is empty, 
has directed cycles or not.  Our complexity results concerning these decision problems are the best known so far; and with respect to space complexity they might be optimal.

The first algorithmic description of all solutions to a given equation over a free group is due to Razborov \cite{raz87,  raz93}. 
His description became known as a \emph{Makanin-Razborov diagram}, and this concept plays a major role in the positive solution of Tarski's conjectures about the elementary theory in free groups 
\cite{KMIV06, sela13}. While Makanin-Razborov diagrams are also graphs whose edges are labeled by morphisms, these morphisms are group homomorphisms, and it is unfeasible to use this approach to directly obtain solutions in freely reduced words, as the cancellation within group elements after applying a homomorphism cannot be controlled. Also, it is extremely complicated to explicitly produce a Makanin-Razborov diagram for a given equation, and this has been done only in very few cases (\cite{MR2542213}).

A description of solution sets as EDT0L languages
was known before only for quadratic word equations over a free monoid by \cite{FerteMarinSenizerguesTocs14}; 
the recent paper 
\cite{DiekertJP2014csr}
 did not aim at giving such a structural result. The present paper builds on the techniques in \cite{DiekertJP2014csr}, 
 in particular we make use of Je\.z's {\em recompression} method \cite{Jez16jacm}.
There is also a description of all solutions for a word equation over free monoids by Plandowski in~\cite{Plandowski06stoc}.
His description is given by some graph which can be  computed in singly exponential time, but without the aim to give any formal language characterization. 

In this paper we restrict ourselves to equations in free groups or free monoids with involution, and their solution sets in reduced words. It is possible to generalize our construction 
 in several directions. 
First, we can replace the free group by any finitely generated free product $\P=\star_{1 \leq i \leq s}F_i$ where each $F_i$ is either a free or finite group, or a free monoid with arbitrary involutions.  
Second, we can allow arbitrary rational constraints for free products. We consider Boolean 
formulae $\Phi$, where each atomic formula is either an equation or a {\em rational constraint}, written as
$X \in L$, where $L\sse \P$ is a rational subset. More concretely, 
let $\P$ be a free product as above, $\Phi$ a Boolean formula over equations and rational constraints, and
$\os{X_1,\cdots, X_k}$ any subset of variables.
Then the techniques developed in this paper allow us to prove that $\cSol(\Phi)= \set{\sig(X_1)\# \cdots \#\sig(X_k)}{\sig \ \text {solves $\Phi$ in reduced words}}$ is EDT0L. Moreover, there is an algorithm which takes 
 $\Phi$ as input and produces an NFA $\cA$ 
 such that $\cSol(\Phi)= \set{\phi(\#)}{\phi \in L(\cA)}$. 
The algorithm is nondeterministic and  uses  quasi-linear 
space in the input size  of  $\Phi$.
However, these more technical results are not the scope of the present paper.
They follow  from standard results in the literature and they have been announced in the conference version of this paper which was presented at ICALP 2015, Kyoto (Japan), July 4 -- 10, 2015 \cite{CiobanuDEicalp2015}. Full proofs are in the corresponding paper  on  arXiv \cite{CiobanuDEarxiv2015}.

\subsection*{Article organisation}

In \prref{sec:prelim} we give
preliminary definitions and notations. In \prref{sec:mainresults} we state the main result, \prref{thm:alice}, that solutions in reduced words to equations in either a free group or a free monoid with involution are described by a finite graph or nondeterministic finite automaton (NFA) which can be constructed in nondeterministic quasi-linear space.  The main work of the paper is in \prref{sec:alicemon} which treats the monoid case. We define the NFA in subsection \ref{subsec:nfacA}, and present the proofs that the NFA encodes only  correct solutions (soundness), and all solutions (completeness), in subsections \ref{subsec:sound}  and \ref{subsec:compness}, respectively.  
The most complicated part is the completeness proof, which  involves producing  a path for a given solution from initial to final node  by alternatively expanding and compressing the equation, ensuring that at all times the size of the equation is bounded so that we stay within the graph. 

Once the monoid case is proved, in \prref{sec:alicegroupie} we follow relatively standard methods to reduce the problem of finding solutions in reduced words in a free group to the monoid case.
In the final section we give an explicit example of the alternating expansion-compression procedure.

We stress that the complicated part of the paper is to prove that the NFA we construct encodes exactly all solutions; the specification and construction of the NFA, and hence the EDT0L language description, is extremely simple by contrast.

\section{Preliminaries}\label{sec:prelim}
\subsection{Monoids with involution}\label{subsec:mwi}

An \emph{alphabet} is a finite set whose elements are called 
{\em letters}. 
By $\Gam^*$ we denote the free monoid over the finite set
$\Gam$.  The elements of a free monoid are called 
{\em words}, and the empty word is denoted by $1$. The length of a word $w$ is denoted by $\abs w$, and ${\abs w}_{x}$
counts how often a symbol $x$ appears in $w$.
Let $M$ be any monoid and $u,v\in M$. We write 
$u \leq v$ if $u$ is a {\em factor} of $v$, which means we can write
$v= xuy$ for some $x,y \in M$. We denote the neutral element in $M$ by $1$, and use the notation $\id{C^*}$ for the neutral element in the monoid of \Endos over a free monoid $C^*$.

An \emph{involution} on a set $\Gam$ is a  mapping $x \mapsto \ov x$ such that 
$\overline{\overline{x}} = x$ for all $x\in \Gam$. For example, the identity map is an \invol. An \emph{involution on a monoid} must also satisfy $\overline{xy}=\overline{y}\,\overline{x}$.
Any involution on a set $\Gam$ extends to  $\Gam^*$: for a word $w = a_1 \cdots a_m$ we let  $\ov{w} = \ov{a_m} \cdots \ov{a_1}$; then $\Gam^*$ endowed with the involution is called a {\em free monoid with involution}. 
If $\ov a = a$ for all $a \in \Gam$ then $\ov{w}$ is simply the word $w$ read from right-to-left. 

A {\em \morph} between sets with involution is a mapping  respecting the involution, and a \morph between monoids with \invol is a \hom $\phi: M \to N$ such that 
$\phi(\ov x) = \ov{\phi(x)}$. A morphism  is a {\em $\Del$-\morph} if $\phi(x) = x$ for all $x \in \Del$ where $\Del\sse M$.
In this paper, whenever the 
term ``\morph'' is used, it refers to a mapping which respects the underlying structure, including the \invol. 
All groups are monoids with involution given by 
$\ov x = x^{-1}$; and all group \homs are \morph{s}.

\subsection{Free partially commutative monoids}\label{subsec:fpcm}
Let ${\Del}$ be a finite set with involution. An \emph{independence relation}
 is an irreflexive  relation $\theta \sse {\Del} \times {\Del}$ 
such that $(x,y)\in \theta \iff  (\ov x,\ov y)\in \theta$. Every independence  relation
defines a \emph{free partially commutative monoid with \invol} $M({\Del},\theta)$ by
$$M({\Del},\theta) = {\Del}^*/\set{xy=yx}{(x,y)\in \theta}.$$
These monoids are well-studied in computer science as they form the basic algebraic model for concurrency, see \cite{dr95, kel73, maz77}. In mathematics free partially commutative groups are commonly referred to as right-angled Artin groups (RAAGs). Their study has a long history with strong connections to topology and geometric group theory, see for example \cite{Wise2012}.

In this paper we will need algorithms for equality and factor testing in free partially commutative monoids. This can be done very efficiently: for example, there is a linear time algorithm (\cite{mes97}) to decide on input $u,w\in {\Del}^*$ 
whether $u\leq w$ in $M({\Del},\theta)$. Here we need the uniform version, as follows: 
the input is a tuple $({\Del},\theta, u,w)$ with $u,w \in {\Del}^*$, and the question is whether $u$ is a factor of $w$ in $M({\Del},\theta)$.
This problem can easily be solved in nondeterministic linear space (which suffices for our purposes) by the following argument:  
first find words $p,q \in {\Del}^*$ by scanning $w$ from left to right
and for each position guessing (nondeterministically) whether each corresponding letter belongs to $p, u$ or $q$, requiring that $\abs {puq}= \abs w$ 
(we do this by marking each letter of the input, which requires linear space).
Second, check that the choice of positions assigned to $u$ produces a word that is indeed equal to $u$. Third, check whether $puq$ is equal to $w$ in $M({\Del},\theta)$. For both the second and third steps we use the ``projection lemma'' of \cite{kel73, cl85}: for example, in the  third step we check that $\abs{puq}_a = \abs{w}_a$ for all $a \in {\Del}$,  then we check that 
 the projections of $puq$ and $w$ to $\os{a,b}^*$ yield identical words for all $a,b\in {\Del}$ such that $ab \neq ba$ in $M({\Del},\theta)$.  The projections are obtained by ignoring all letters in $puq$ and $w$ 
which are not in  $\os{a,b}$. 

Another fact about partially commutative monoids that we use later is that for 
$u\in M({\Del},\theta)$ the values $\abs u$ and ${\abs u}_{a}$ are well-defined 
since ${\abs {xy}}_{a} ={\abs {yx}}_{a}$ for all $x,y \in {\Del}^*, a \in {\Del}$. 

We will define free partially commutative monoids through 
``types'' in \prref{subsec:typtheta}, which for simplicity of notation are also  denoted by $\theta$.

\subsection{Languages}\label{subsec:Lang}
\emph{Languages} refer traditionally to subsets of finitely generated free monoids; the class of \emph{regular languages} can be defined via rational expressions, nondeterministic finite automata, 
or recognizability via \hom{s} to finite monoids, to mention just a few of the possible definitions \cite{pin86}. 
These notions generalize to arbitrary monoids, but  lead to different classes, in general.  

We define a \emph{rational subset} in any monoid $M$ by means of  \emph{nondeterministic finite automaton}, NFA for short. 
An NFA is a directed finite graph $\cA$ with initial and final {\em states}, where the
transitions between states are labeled by elements of the monoid $M$. We say that 
$m\in M$ is \emph{accepted} by the automaton $\cA$ if
there exists a path from some initial to some final state such that multiplying the edge labels together in $M$ yields $m$. This defines the accepted 
language $L(\cA) = \set{m \in M}{m \text{ is accepted by }\cA}$. 
Then  $L\sse M$ is \emph{rational} \IFF 
$L$ is accepted by some NFA over $M$ (see \cite{eil74}). 
An NFA is called \emph{trim} if every {state} is on some path from an initial to a final {state}. For a trim NFA $\cA$ we have 
$L(\cA)\neq \es$ \IFF $\cA \neq \es$. 

 We say that $L\sse M$ is \emph{recognizable} if there is a
\hom $\nu: M\to N$ to a finite monoid $N$ such that $L = \oi{\nu}(\nu(L))$. 
The family of recognizable subsets is closed under finite union and complementation (and therefore also under finite intersection), and therefore forms a Boolean algebra. 
For finitely generated free monoids Kleene's Theorem asserts that  a subset is recognizable \IFF it is \emph{rational}; and in this context a rational subset is also called {\em regular}.

In this paper we are mainly interested in rational subsets of free groups $F(\Apos)$, free monoids $A^*$, and monoids $\End(C^*)$ of \Endos over a free monoid $C^*$. If $\abs C \geq 2$, then $\End(C^*)$ is neither free nor finitely generated  and it contains non-trivial finite subgroups.

Suppose we have an NFA where each \tra label is an \Endo in $\End(C^*)$ which is applied in the opposite direction of the transition. If a path  is labelled by the sequence $h_1, \ldots , h_t$, then we can apply 
the \Endo $h=h_1 \cdots h_t$ to an element $u\in C^*$ and the result is 
a word $h(u)=h_1 \cdots h_t(u) \in C^*$. Thus, 
$\set{h(u)}{h \in L(\cA)}$ defines a language in $C^*$. 
This leads to 
the notion of EDT0L, defined next.
 
 \subsubsection{EDT0L Languages}\label{subsec:edtollang}
 The acronym {EDT0L} refers to {\em {\bf E}xtended, {\bf D}eterministic, {\bf T}able, 
{\bf 0} interaction, and {\bf L}indenmayer}. 
There is a vast literature on Lindenmayer systems, see \cite{RozS86}, with various  acronyms such as D0L, DT0L, ET0L, HDT0L and so forth. For more background on Lindenmayer systems 
we refer to \cite{rs97vol1}.
The subclass EDT0L is equal to HDT0L (see for example \cite[Thm.~2.6]{rs97vol1}), and has received particular attention. 
It is a subclass of indexed languages 
in the sense of Aho \cite{Aho68}, see for example
\cite{EhrRoz77}. Indexed languages are context-sensitive, and they strictly contain
all context-free languages. The classes of EDT0L  and 
context-free languages are incomparable \cite{EhrRoz77} and therefore the inclusion 
of EDT0L  into indexed languages is proper. 

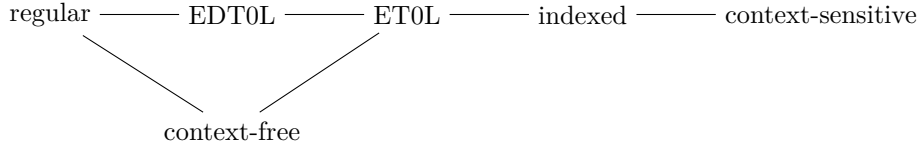
\begin{figure}[htbp]
\begin{center}
  \begin{tikzpicture}[ node distance = 30pt]

 \node (reg) {regular};
  \node (edt) [ right=of reg] {EDT0L};
  \node (cf) [ below=of edt] {context-free}; 
  \node (etol) [ right=of edt] {ET0L};  
    \node (ind) [ right=of etol] {indexed};
  \node (cs) [ right=of ind] {context-sensitive};

    \draw(reg) --  (edt);
        \draw (edt) --  (etol);
    \draw (reg) -- (cf);
            \draw (cf) --  (etol);
                  \draw (etol) --   (ind);
             \draw  (ind) -- (cs);
 \end{tikzpicture}
  \caption{Containments of formal language classes. Each edge from left to right represents strict containment.
 }\label{fig:language_diagram}
\end{center}
\end{figure}

We define EDT0L languages in $A^*$ through a characterization  (using  rational control)
due to Asveld \cite{Asveld1977}, which is the analogue of  Ginsburg and Rozenberg's result for ET0L languages (\cite[Lem.{} 4.1]{GinsburgRoz75}).
We start with some alphabet $C$ such that $A\sse C$,
and a rational set of \Endos $\cR \sse \End(C^*)$.
Note that if $\cR \sse \End(C^*)$ is any subset of \Endos, then 
we can apply $\cR$ to any word $u\in C^*$  and 
we obtain a subset $\set{h(u)}{h \in \cR} \sse C^*$.

\begin{definition}\label{def:edt0lasfeld}
Let $A$ be an alphabet and $L\sse A^*$. 
We say that $L$ is an {\em EDT0L} language if there is an alphabet $C$ with 
$A\sse C$, a 
rational set of \Endos $\cR \sse \End(C^*)$, 
and a letter $c\in C$ 
such that 
$L = \set{{h}(c)}{{h} \in \cR}.$ 
\end{definition}
The set $\cR$ is called the {\em rational control}, and 
$C$ the {\em extended alphabet}.

Note that for an arbitrary set $\cR$ of endomorphisms of $C^*$ we have $\set{{h}(c)}{{h} \in \cR} \sse C^*$, but the definition implies that $\cR$ must guarantee ${h}(c)\in A^*$ for all 
${h} \in \cR$.

\begin{example}\label{ex:edtvd}
Let $A= \os{a,b}$ and $C= \os{a,b, \#}$. Consider four endomorphisms
$f,g_a,g_b,h$ defined as $f(\#) = \#\#$, 
$g_a(\#) = a\#$, $g_b(\#) = b\#$, and $h(\#) = 1$, and on all other letters
$f,g_a,g_b,h$ behave like the identity. 
 Consider the rational language   
$\cR = h\os{g_a, g_b}^* f$ (where endomorphisms are applied right-to-left).
A simple inspection shows  that $\set{\phi(\#)}{\phi \in R} = \set{vv}{v\in A^*}$, which is not context-free.  
\end{example}

\subsection{Complexity}\label{subsec:complex}
We use the standard $\Oh$-notation for functions from $\N$ to $\mathbb{R}_{\geq 0}$. A function $f$ is called \emph{quasi-linear} if $f(n) \in \Oh(n \log n)$. We say that $f$ is \emph{singly exponential}
if $f(n) \in 2^{\Oh(p(n))}$ where $p(n)$ is a polynomial. We also use the standard meaning of complexity classes like $\NP$, $\NSPACE(f)$, $\DSPACE(f)$ and $\DTIME(f)$ as in \cite{pap94}.

Let $\cC$ and $\cD$ two domains and for each $x \in \cC \cup \cD$ we let 
$\gen x \in \os{0,1}^*$ denote some binary encoding. We assume that for every $x\in \cC$ its \emph{input size} is defined as a natural number which might be different from the binary length of $\gen x$.
For example, in our case we define the input size of an equation over a free group or monoid to be the  length of the equation plus  
 the number of generators of the group or monoid. As usual, we omit details on the specific encoding and how to check that a binary string $y$ is of the form 
 $y = \gen{x}$ for some $x\in \cC$. In our case, 
 we content ourselves that the encoding of a word of length $n$ over some alphabet $\Gam$ uses at most $\Oh(n \log \abs \Gam)$ bits and that the 
 check  $y = \gen{x}$  can be done deterministically in linear space with respect to the binary length of $y$.

A function $t:\cC \to \cD$ is computable in 
$\NSPACE(f)$ if there is a nondeterministic Turing machine $M$ with a two-way read-only input tape, a work tape, and a write-only output tape. The input $x \in \cC$ is given as the binary string $\gen x$. 
During the computation the machine writes some binary string on the output tape from left to right such that for the entire computation the size of $M$'s work tape is bounded 
by $\Oh(f(n))$ where $n$ is the input size of $x$. There must be at least one run of the machine where $M$ stops and if $M$ stops, then output must be the correct value
$\gen{f(x)}$. 
We rely on a result by Immerman and Szelepcs{\'e}nyi which implies that $\NSPACE(f)$ is (effectively) closed under complementation 
for functions $f$ satisfying $\log n \in \Oh(f(n))$ \cite[Theorem 7.6]{pap94}). As a consequence, 
``trimming'' an automaton will become possible in $\NSPACE(n\log n)$ in \prref{subsec:trimcA}.
Recall that every $\NSPACE(n\log n)$-computable function can also be simulated by some deterministic algorithm in time $2^{\Oh(n\log n)}$ (see  \cite[Theorem 3.3]{pap94}).

\subsection{Word equations over monoids with rational constraints}\label{subsec:wemrnf}

Let $A$ be an alphabet of \emph{constants} with involution and let $\pi: A^* \to M$ be a surjective \morph onto a monoid with \invol $M$. Furthermore, let $\cX$ be a set of {\em variables}. 
We may assume that $\cX$ is endowed with an \invol without fixed points. Thus, $X\neq \ov X$ for all $X\in \cX$. 
\begin{definition}\label{def:WE}
A \emph{word equation with rational constraint over $M$} is a pair $(U,V)$ of words 
$U,V\in (A \cup \cX)^*$ which has the following attributes.
\begin{itemize}
 \item 
The \emph{input size} of the equation is defined as $\abs{A} + \abs {UV}$.
\item The rational constraint is given by a \hom $\nu:(A\cup \cX)^*\to N$, where $N$ is a finite monoid. 
\item A \emph{\solu} of the equation $(U,V)$ with constraint $\nu$
is given by a map $$\sig: \cX \to A^*$$
which extends to a \hom  $\sig: (A \cup \cX)^* \to A^*$ that fixes the constants, such that for all $X\in \cX$: 
\begin{enumerate}
\item $\sig(\ov X) = \ov{\sig(X)}$, \hfill {\em i.e.} $\sig: \cX \to A^*$ is a \morph,
\item $\nu(X) =\nu\sig(X)$, \hfill {\em i.e.}   
the solution respects the constraint on $X$,
\item $\pi\sig(U)= \pi\sig(V)$, \hfill {\em i.e.} $\sig(U)$ and $\sig(V)$
are equal in the monoid $M$.
\end{enumerate}
\end{itemize}
\end{definition}

Note that we constrain the solutions to be in a recognizable set (see the definitions in \prref{subsec:Lang}), but in this case the notions of recognisable and rational sets are the same,
since we are in the free monoid $(A \cup X)^*$.

\section{Solution sets for equations over free monoids with involution and free groups: the main results} \label{sec:mainresults}

Let $\Apone = \Apos \cup \set{\ov a}{a \in \Apos}$ be a finite alphabet with \invol and assume that 
the involution is without fixed points: $\ov a \neq a$ for all $a \in \Apone$. We let $\FGA$ be the free group over $\Apos$ and we realize the involution inside $\FGA$  by $\ov a = \oi a$. 
Thus
$$\Apone = \Apos \cup \set{\oi a}{a \in \Apos} \sse \FGA \sse \Apone^*.$$
Following standard terminology,  a word $w\in \Apone^*$ is  {\em reduced} if it does not contain any factor $a \ov a$ where $a \in \Apone$. The set 
of reduced words is a regular subset $\F\sse \Apone^*$ which is closed under \invol. We fix  $\F$ as a set of normal forms for $\FGA$; thus, as a set, we identify $\FGA$ with $\F$. 
The inclusion $\Apone \sse \FGA$ induces the canonical projection $\pi: \Apone^* \to  \FGA$. Given a word $w$ we obtain $\pi(w)$ by a repeated cancellation of all factors $a \ov a$; and $w$ is reduced \IFF $\pi(w) = w$. 

We shall also use a special symbol $\#$ which is not in $\Apone$ and serves as ``marker''. 
For example, we will encode a system of equations
$\set{(U_i,V_i)}{1 \leq i \leq s}$ as a single
equation
\begin{equation}\label{eq:syseq}
(U_1 \# \cdots \# U_s, V_1 \# \cdots \# V_s).
\end{equation}
If we require that no $\sig(X)$ is allowed to use $\#$, where $X$ is a variable, then 
\begin{equation}\label{eq:syseqtwo}
\forall i: \pi\sig(U_i) = \pi\sig(V_i) \iff \pi\sig(U_1 \# \cdots \# U_s)=\pi\sig( V_1 \# \cdots \# V_s)\end{equation}
since positions of the $\#$ letters must be the same on both sides.
In our context, rational constraints are the most convenient way to ensure that no $\#$ appears in $\sig(X)$, see \prref{subsec:mN}.
We let 
$$ A = \Apone \cup \os{\#}$$
with $\ov \# = \#$. Thus, $\os{1,\#}$ forms a group which is isomorphic to $\Z /2 \Z$ if we let $\oi \# = \#$.

In order to have a uniform statement we let 
$\MMA$ be either the free monoid with involution $A^*$ or the free product of the free group $\FGA$ with the cyclic group $\os{1,\#}$ of order $2$. 
Thus, henceforth:
$$\MMA = A^* \quad \text{ or }  \quad \MMA = A^*/\set{a \ov a = 1}{a \in A},$$  and
$\pi: A^* \to \MMA$ is the canonical projection induced by the inclusion $A \sse \MMA$. In both cases $\pi$ is injective on $\F \sse A^*$, and if $\MMA = A^*$, then $\pi$ is just the identity.  

Given a word equation $(U,V)$ with $UV \in (\Apone\cup \cX)^*$  over $\MMA$, 
we say that a \solu $\sig$ is a \emph{solution in reduced words} if 
$\sig(X) \in \F$ for all $X\in \cX$. We will realize this condition as a 
rational constraint $\mu$ into a finite monoid $N$ with a zero element $0\in N$ such that $\mu(w) \neq 0$ \IFF $w \in \F$.

\begin{theorem}\label{thm:alice}
Let $(U,V)$ be an equation over $\MMA$ of input size $n = \abs A + \abs{UV}$ (according to \prref{def:WE}) 
and in variables $X_1, \ov{X_1} \lds X_m, \ov{X_m}$.   
Then 
 there is an $\NSPACE(n \log n)$ algorithm 
which computes $c_1 \lds  c_m \in C$, where $C \supseteq A$ is an extended alphabet of size  $\abs C \in \Oh(n)$, and a trim NFA $\cA$ which produces the set of solutions in reduced words. That is,
\begin{equation}\label{eq:alice} \begin{split}
 \{(\sig(X_1) \lds  \sig(X_m))&
 \in \F\times \cdots \times \F\mid \pi\sig(U)=\pi\sig(V)\} \\ 
 =  \{(h(c_1) \lds  h(c_m))&\in C^*\times \cdots \times C^*\mid h\in L(\cA)\}.
 \end{split}
\end{equation}
  The NFA has the following properties.
 \begin{enumerate}
\item It is nonempty \IFF the equation $(U,V)$ has some solution. 
\item It has a directed cycle \IFF $(U,V)$ has infinitely many solutions. 
\end{enumerate}
These properties can also be decided in $\NSPACE(n \log n)$. 
\end{theorem}

Recall that the input size $n$ used in the statement of the theorem might be  smaller than the length of some binary encoding for the input. 
If the number of distinct symbols used in the equation is constant, then our algorithm is quasilinear in the input size; if, on the other hand, the number of distinct symbols used in the equation is linear, then we need linear space, only.

Theorem \ref{thm:alice} yields the characterization of solutions sets as \edtol languages. To do so, we identify a tuple of words
$(w_1 \lds w_k)\in \F$ with the single word 
 $w_1\# \cdots \#w_k\in A^*$.

 Let $(U,V)$ be an equation as in \prref{thm:alice}. For any subset $\os{Z_1 \lds Z_k}$  of variables appearing in $UV$ we define the solution set  
as 
\begin{equation}\label{eq:solset}
\cSol_\cZ(U,V)= \set{\sig(Z_1)\# \cdots \#\sig(Z_k)}{\sig \text { solves $(U,V)$ in reduced words}}.
\end{equation}
Note that for $k = 0$ we have $\cSol_\es(U,V)=\es$ if
the equation $(U,V)$ has no solution and $\cSol_\es(U,V)=\os 1$ otherwise. 
Considering subsets of variables allows for some flexibility. In particular, 
we can introduce auxiliary variables which do not impact the solution set. If, however, every variable occurring in $UV$ is either of the form $Z_i$ or $\ov{Z_i}$ for some $1 \leq i \leq k$, then we say that 
$\cSol_\cZ(U,V)$ is a \emph{full solution set}.  

\begin{corollary}\label{cor:alice}
Let  $(U,V)$ be an equation as in \prref{thm:alice} and let $\os{Z_1 \lds Z_k}$  be any subset of variables appearing in $UV$. 
Then $\cSol_\cZ(U,V)$ is an \edtol language. More precisely, if $\cA$ is the trim NFA constructed in \prref{thm:alice}, then we can find $c'_1, \ldots, c'_k\in C$ such that
$$\cSol_\cZ(U,V)= \set{h(c'_1\# \cdots \#c'_k)}{h\in L(\cA)}.$$ 
In particular, the full solution set is \edtol.
\end{corollary}
\begin{proof}
The language characterization follows from the Definition \ref{def:edt0lasfeld} of an \edtol language, given that each $Z_j$ corresponds to some $X_i$ in \prref{thm:alice}.
\end{proof}

Note that 
\prref{thm:alice} shifts the traditional  perspective from 
 solving an equation to  an effective construction of some NFA
 producing an EDT0L set. Once the NFA is constructed, the existence of a solution, or whether the number of solutions in reduced words is zero, finite or infinite, 
become graph properties of the NFA. Thus, the algorithmic difficulty of solving equations and describing their solution set reduces to the complexity of building a nondeterministic finite automaton for a given input.

\section{Proof of Theorem~\ref{thm:alice} in the monoid case: $\M(A) = A^*$}\label{sec:alicemon}

In this section we prove Theorem~\ref{thm:alice} in the monoid case. Before delving into the  proof, we introduce in Subsections \ref{subsec:tie}--\ref{subsec:traG} further necessary terminology and notation.  

Let $\MMA=A^*$. In this case $\pi= \id{A^*}$ and so $\pi$ is not needed in the rest of this section.
Without restriction, we may assume $\abs{\Apos} \geq 1$. 

Let
$\Xinit = \os{X_1,\ov{X_1} \lds X_m,\ov{X_m}}$ be the initial set of variables, that is,   for each $1\leq i\leq m$ either $X_i$ or $\ov{X_i}$ occur in $UV$.

Let $\kappa \in \Oh(1)$ be some ``large enough'' constant, whose exact value will be discussed in \prref{subsec:sizeC}, and 
 choose 
an alphabet $C$ of {\em constants} and an alphabet $\OO$ of {\em  variables} such that 
$$C\supseteq A, \abs C = \kappa \cdot n \text{ and } \OO\supseteq\Xinit\ , \abs \OO = 6n.$$   
Fix $\Gam= C \cup \OO$. 
We assume that $C$ and $\OO$ are sets with involution and that, inside $\Gam = C \cup \OO$, the marker $\#$ is the only self-involuting symbol. Thus, $\ov \# = \#$ and $\ov x\neq  x$ for all $x\in \Gam\sm \os{\#}$.

By $\Sig$ we denote the set of $C$-\morph{s} $\sig: \Gam^* \to C^*$. 
Every solution will be drawn from $\Sig$.

\subsection{The initial word equation $\Winit$}\label{subsec:tie}
For technical reasons we need that for every variable $X_{i}$ which appears in $UV$ there is some factor $\#X_i\#$ appearing 
in the initial equation. 
Instead of viewing equations as equalities between two words $U$ and $V$, we will treat equations as a statement about a single word $W \in \Gam^*$, as follows. This will require us to redefine the notion of \solu as well. 

We define the {\em initial equation} $\Winit\in (A \cup \Xinit)^*$  as: 
\begin{equation}\label{eq:Winit}
\Winit=  \#X_1\#\cdots \#X_m \# U\# V\#\ov {U}\# \ov {V}\# \ov{X_m}\#\cdots \#\ov{X_1}\#.
\end{equation}

Then for every $\sig \in \Sig$ we have 
$$\sig(U)=\sig(V) \iff \sig(\Winit)=\sig(\ov \Winit)$$
and 
  \begin{align*}
 \{(\sig(X_1) \lds  \sig(X_m))&
 \in \F\times \cdots \times \F\mid \sig \in \Sig \wedge \sig(U)=\sig(V)\} \\
 =  \{(\sig(X_1) \lds  \sig(X_m))&
 \in \F\times \cdots \times \F\mid  \sig \in \Sig \wedge \sig(\Winit)=\sig(\ov \Winit)\}.
\end{align*}
We have the following symmetry: 
if  $w\leq \Winit$ is a factor and no $\#$ appears in $w$, then 
 $\ov w\leq \Winit$, too.
The number of $\#$ letters in $\Winit$ is odd, and there is a distinguished $\#$ exactly in the middle of $\Winit$.  

Observe that $\Winit$ is longer than $UV$, but clearly linear in  $n$.
More concretely, since $m \leq \abs{UV}$ and $n=\abs{UV}+\abs{A}>\abs{UV}+1$, we get the bound:
\begin{equation}\label{eq:UVwinit}
\abs \Winit \leq 4m+5+2\cdot \abs{UV}\leq 6\cdot\abs{UV} +  5 < 6(\abs{UV}+1)< 6n.
\end{equation}
Also observe that $\sum_{X\in \Xinit}\abs{\Winit}_X \leq 2m+2\abs{UV}\leq 4n$.

\subsection{The finite monoid $N_\F$}\label{subsec:mN}
In order to ensure that solutions are in reduced words which do not contain the symbol $\#$, 
 we introduce a morphism to a fixed finite monoid $N_\F$ which 
plays the role of (a specific)
rational constraint. 
We define $N_\F$ as follows: $N_\F = \os{1,0} \cup (\Apone \times \Apone)$ with multiplication given by $1\cdot x = x \cdot 1 = x$, 
$0\cdot x = x \cdot 0 = 0$, and 
\begin{equation*}
\begin{array}{llllll}
(a,b)\cdot (c,d) = \left\{\begin{array}{llllll} 
(a,d) &&  \mathrm{if }\;  b\neq \ov c\\
\; 0 &  & \mathrm{otherwise.}
\end{array}\right.\end{array}
\end{equation*}
The monoid $N_\F$ has a natural involution given by $\ov 1 = 1$, $\ov 0 = 0$, and 
$\ov{(a,b)} = (\ov b , \ov a)$.

The \morph{s} to $N_\F$ are 
defined on subsets of $\Gam$, and although they change during the algorithm, they always extend the following fixed 
\morph $$\mu_0:A^* \to N_\F$$ which is defined  by $$\mu_{0}(\#) =0, \ \ \mu_{0}(a) = (a, a)$$ for $a 
\in \Apone$. It is clear that $\mu_{0}$ respects the involution and $\mu_{0}(w) = 0$ if and only if
either $w$ contains $\#$ or $w$ is not reduced.
If, on the other hand, $1 \neq w\in \Apone^*$ is reduced, then 
$\mu_{0}(w) =(a,b)$, where $a$ is the first and $b$ the last letter of $w$. 
An additional feature is that $\mu(w) = 1$ \IFF $w$ is the empty word. 

Defining $\mu(X)$ for a variable $X$ has the following meaning for a \solu 
$\sig$ with $\sig(X) \in \Apone^*$: 
the value $\mu(X) = 0$ is not possible in any \solu, $\mu(X) = 1$ implies $\sig(X) = 1$, and 
$\mu(X) = (a,b) \iff \sig(X) \in \F \cap  a\F \cap \F b$.

\subsection{Types}\label{subsec:typtheta}

Later in the proof we will need to perform compression of large blocks of letters in an efficient manner. This will be achieved by putting a partially commutative structure on the monoid we work with. The partial commutativity will be induced by \emph{types}, which we introduce below. The basic idea is that we assign a variable $X$ the ``type" $\theta(X) =c$ when we predict that in some solution
$\sig(X) \in c^*$ (so $X$ and $c$ commute), and we assign a constant $b$ the ``type" $\theta(b)=c$ when we rename some letters $b$ as $c$.

Besides the initial alphabet $A$ and the global alphabet $C$, we also need a \emph{current} alphabet of constants $B$, where $A \sse B= \ov B \sse C$, and a \emph{current} set of variables $\cX = \ov \cX \sse \OO$.
Let $\Del =B\cup \cX$. A \emph{type} is a partially defined 
function $\theta:(\Del\sm A) \to (B\sm A)$
which respects the involution. We identify  $\theta$ with the relation 
$\set{(\theta(x), x) \in \Del\times \Del}{\theta(x) \text{ is defined}}$.
We obtain an independence relation
$${\theta} = \set{(\theta(x),x)\in \Del \times \Del}{\theta(x) \text{ is defined for } x}$$ 
and hence a free partially commutative monoid 
$$M(\Del,{\theta})= \Del^*/\set{x\theta(x)= \theta(x)x}{\theta(x) \text{ is defined for } x}.$$ 
If the domain where $\theta$ is defined is empty, then 
$M(\Del,{\theta})= M(\Del,\es)$ is the free monoid $\Del^*$. 
\begin{remark}\label{rem:thetalength}
By definition, the size $\abs \theta$ is bounded by $\abs \Del$. Hence, it is linear in $n$ and the specification of $\theta$ needs $\Oh(n \log n)$ bits.
\end{remark}

\begin{definition}\label{def:monstruc}
Let $B$ satisfy $A \sse B= \ov B \sse C$,  
$\cX = \ov \cX \sse \OO$, and $\theta$ be a type. 
The notation
$$M(B,\cX,\theta,\mu)$$ denotes the free partially commutative monoid with \invol
$M(B \cup\cX,\theta)$, 
 equipped with a \morph $\mu: M(B \cup\cX,\theta) \to N_\F$ such that $\mu(a) = \mu_0(a)$ for all $a \in A$, where $\mu_0: A^* \to N_\F$ is the \morph specified in \prref{subsec:mN}. 
We call $M(B,\cX,\theta,\mu)$ a \emph{structured monoid}.

A \emph{\morph} $\phi$ from $M(B,\cX,\theta,\mu)$ to $M(B',\cX',\theta',\mu')$
is a \morph of monoids with \invol $\phi:M(B,\cX,\theta,\mu)\to M(B',\cX',\theta',\mu')$ such that $\mu' \phi = \mu$.
\end{definition}

\prref{def:monstruc} implies that 
whenever $\theta(x)$ is defined, then $\mu(x\theta(x))= \mu(\theta(x)x)$
(because $\mu$ is a \hom). Henceforth we use the following conventions. If $B'\sse B$ and $\cX'\sse \cX$ with $A \sse B'= \ov {B'}$ and
$\cX' = \ov {\cX'}$, then  
$M(B',\cX',\theta,\mu)$ denotes the structured monoid 
$M(B',\cX',\theta',\mu')$ where $\theta'$ and $\mu'$ are induced by the restrictions 
of $\theta$ and $\mu$ to $B'\cup \cX'$.  Moreover, if $M(B,\cX,\theta,\mu)$ is known from the context, then we abbreviate
 $M(B,\es,\theta,\mu)$ as $M(B)$. Since no letter from $A$ is involved in a type, $M(A)$ is the free monoid with \invol $A^*$ together with the \morph 
$\mu_0: A^* \to N_\F$, and
$$M(A) = M(A,\es,\es,\mu_0) \sse M(B) \sse M(B,\cX,\theta,\mu) \arc \mu N_\F.$$

\subsection{Reference list of symbols}\label{subsec:los}
In Table~\ref{table:los} we summarise notations introduced so far for easy reference. These conventions hold unless 
stated otherwise.
They also apply to ``primed'' symbols such as $B'$, where $B'$ denotes
a set with $A \sse B'= \ov{B'} \sse C$. 
\begin{table}[h!]
\begin{tabular}{|l|}
\hline $A_+ \sse \Apone$, the initial alphabets without self-involuting letters.
\\ \hline $\Apone \cup \os \# = A \sse B= \ov B \sse C$.
\\ \hline $\Gam = C \cup \OO$  and $x = \ov x\in \Gam$ implies $x= \#$. 
\\ \hline $\cX = \ov \cX \sse \OO$, the current set of variables.
\\ \hline  $n=\abs{A} + \abs{UV}$, $\abs{C} = \kappa n$ and $\abs{\OO} = 6n.$  
\\ \hline $\Del= B\cup \cX$.
\\ \hline $\mu:\Del\to N_\F$, a \morph with $\mu(a) = \mu_0(a)$ for $a \in A$.
\\ \hline $\theta:(\Del\sm A) \to (B\sm A)$, the type defining an independence relation.
\\ \hline $M(\Del,{\theta})$, free partially commutative monoid defined by $\Del$ and $\theta$.
\\ \hline $M(B,\cX,\theta,\mu)= M(\Del,{\theta})$ together with $\mu$ which extends $\mu_0:A^* \to N_\F$. 
\\ \hline $M(B)$, submonoid of $M(B,\cX,\theta,\mu)$ 
together with the restriction of $\theta$, $\mu$. 
\\ \hline $a,b,c, \ldots$ refer to letters in $C$.
\\ \hline $u,v,w, \ldots$ refer to words in $C^*$.
\\ \hline $X,Y,Z, \ldots$ refer to variables in $\OO$.
\\ \hline $x,y,z, \ldots$ refer to words in $\Gam^*$.\\\hline
\end{tabular}

\caption{Reference list of symbols.}
 \label{table:los}
\end{table}

\subsection{Extended equations and their solutions}\label{subsec:exes}

The states of the NFA we are going to construct correspond to equations derived from our initial equation. Each state contains such an equation, together with the specification of which set of constants, variables and types are used. Moreover, we keep track of the \morph $\mu$ which represents the constraint. Formally, we use the notion of \emph{extended equation}. The notions we introduce now are quite technical, but the reader should keep in mind that the most important fact is that an extended equation contains an equation which is a modification of the initial equation, and this equation has bounded length. When types are present, this equation is an element in a free partially commutative monoid rather than simply a word in a free monoid.

\begin{definition}\label{def:wellf2}
An {\em extended equation} is a tuple $(W,B,\cX,\theta,\mu)$, where $W$ is a word 
in $(B\cup\cX)^*$ such that: 
\begin{enumerate}
\item $\abs W\leq 204 n$. 
\item  If $\theta = \es$, then $\sum_{X\in \cX}\abs{W}_X \leq 4n$.  Otherwise $\sum_{X\in \cX}\abs{W}_X \leq 12n$. 
\item ${\abs W}_{\#} = {\abs \Winit}_{\#}$ and $W\in \#(B\cup \cX)^*\#$. 
\item Every $x$ with $\# \neq x\in B \cup \cX$ satisfies
$\mu(x) \neq 0$.  
\item Every $X\in \cX$ appears in $W$. 
\item If $x \leq W$ is a factor with $\abs{x}_\# =0$, then $\ov x\leq W$, too.
\end{enumerate}
\end{definition}
\begin{remark}
As noted above, the word $W$ (including the notion of factor) is to be seen as representing an element in the free partially commutative monoid $M(B,\cX,\theta,\mu)= M(B\cup\cX,\theta)$. 
Note that by definition $\abs\theta\leq \abs{B\cup \cX}$ (see \prref{rem:thetalength}).
The 
bounds on the length of $W$, and on the number of variables appearing in $W$, will be explained in later sections (\prref{subsec:spaceBlock}), where we will show that we can find all solutions to an input equation by considering modified equations that satisfy these restrictions. What is important for now is that $\abs W\in\Oh(n)$ which means the number of extended equations is finite.
\end{remark}

\begin{definition}\label{def:weightexe}
Let  $V = (W,B,\cX,\theta,\mu)$ be an \exe. The
\emph{weight} $\Abs{V}$ of $V$ is a $4$-tuple of natural numbers,
 $\Abs{V}= (\oo_1,\oo_2,\oo_3,\oo_4)$, where 
\begin{align*}
\oo_1 &=  \abs W, \\
\oo_2 &=  \abs{W} - \abs{\set{a \in B}{{\abs W}_a\geq 1}}, \\
\oo_3 &=  \abs{W} -   \abs{\theta}, \\
\oo_4 &=   \abs B.
\end{align*}
\end{definition}

\begin{remark}We order tuples in $\N^\ell$ lexicographically.
The lexicographic ordering is chosen to function as follows. If we start at an equation of high weight, then the weight of the equation reduces by ``compression''. The first component gives more weight to longer equations. If two equations have the same length, then we declare the equation in which more distinct constants appear to be smaller
because the term $\abs{\set{a \in B}{{\abs W}_a\geq 1}}$ appears with a negative sign.
 If two equations have the same length and use the same number of distinct constants, we declare the equation in which more symbols are typed to be smaller.
Finally, if both equations have the same length, the same number of distinct letters in use, and the same number of typed symbols, then we declare the equation defined over the smaller set $B$ to be smaller.
\end{remark}

Since for every extended equation we have a current alphabet $B$, we need the notion of a $B$-solution, which can then be extended to a solution over the desired alphabet $A$. The next few pages are somewhat technical, but will be used to justify that when we modify extended equations in certain ways, solutions are  preserved.

\begin{definition}\label{def:extequat}
Let $V= (W,B,\cX,\theta,\mu)$ be an extended equation.
\begin{itemize}
\item A {\em $B$-\solu} at $V$ is a $B$-\morph $\sig:M(B,\cX,\theta,\mu)\to M(B,\es,\theta,\mu)$ such that
$\sig(W) = \sig(\ov W)$ and $\sig(X) \in y^*$ whenever $(X,y)\in \theta$.
\item A {\em \solu} at $V$ is a pair  $(\alp,\sig)$ where $\sig$ is a $B$-\solu and 
$\alp: M(B,\es, \theta,\mu) \to  A^*$ is an $A$-\morph (which implies $\mu= \mu_0 \alp$). Moreover, if the set $\cX$ in $V$ is nonempty, then we
require that $\alp$ is nonerasing, that is, $\alp(a) \neq 1$ for all $a\in B$. 

\end{itemize}
The \emph{weights} $\Abs{\alp,\sig}$ and $\Abs{\alp,\sig, V}$ of a \solu  $(\alp,\sig)$ at $V$ are defined as 

\begin{align}\label{eq:weightsolu}
\Abs{\alp,\sig} &= \sum_{X\in \cX}\abs{\alp\sig(X)} \in \N\\
\label{eq:weightsoluV}
\Abs{\alp,\sig, V} &= (\Abs{\alp,\sig}, \Abs V) \in \N^5.
\end{align}
\end{definition}

\begin{remark}\label{rem:nova}
Let $V= (W,B,\cX,\theta,\mu)$ be an extended equation with a solution 
$(\alp,\sig)$.  
Then $\sig(X)$ cannot have any factor of the form $\#$ or $a \ov a$ 
with $a\in B$ because $0 \neq \mu(X)= \mu_0\alp \sig(X)$.
In particular, $\alp \sig(X)$ is a reduced word in $\Apone^*$. Hence,
$\alp \sig$ satisfies the constraint $\alp \sig(X)\in \F$.
Note that a priori we don't exclude the possibility that factors $a\ov a$ appear in $W$, since for example it could be that $\Winit$ contains a factor $aX$ and some \solu $\sig(X)$ begins with 
$\ov a$. 
\end{remark}

The next two lemmas show how \morph{s} between structured monoids transform \solu{s} of \exe{s}. These two lemmas will play an important role in the proof of the  algorithm ``soundness''. 

In the first lemma we consider the  \morph{s} which leave all constants invariant, and conclude that such a morphism decreases the weight of a solution. In addition, this lemma specifies a situation, in part (iv), when the weight strictly decreases.

\begin{lemma}\label{lem:varsub}
Let $V= (W,B,\cX,\theta,\mu)$ and $V'= (W',B,\cX',\theta',\mu')$ be \exe{s} 
such that $\theta(a) = \theta'(a)$ and   $\mu(a) = \mu'(a)$ for all $a \in B$.
In other words, $M(B) = M(B,\es,\theta,\mu)= M(B,\es,\theta',\mu').$

Let $\tau: M(B,\cX,\theta,\mu) \to M(B,\cX',\theta',\mu')$ be a $B$-\morph such that
$W'=\tau(W)$ and  $\alp: M(B) \to  M(A,\es,\emptyset,\mu_{0})$ be an $A$-\morph such that $\alp(a) \neq 1$ for all $a\in B$. 
 
Given a $B$-solution $\sig'$ at $V'$,  
define a $B$-\morph $\sig: M(B,\cX,\theta,\mu) \to  M(B)$
by $\sig(X)=\sig'\tau(X)$. 

Then the following assertions hold. 
\begin{itemize}
\item[(i)] $(\alp,\sig)$ is a \solu at $V$ and $(\alp,\sig')$ is a \solu at $V'$. 
\item[(ii)]$\alp\sig(W) = \alp \sig'(W')$.
\item[(iii)] $\Abs{\alp,\sig}\geq \Abs{\alp,\sig'}$.
\item[(iv)] If there is some $X$ with  $\tau(X) \in \cX'^*a\cX'^*$ where $a \in B$ and $\alp(a) \neq 1$, then  
$\Abs{\alp,\sig} > \Abs{\alp,\sig'}$. 
\end{itemize}
\end{lemma}

\begin{proof}
\begin{itemize}
\item[(i)] Since $\sig'$ is a $B$-solution at $V'$ we have 
 $$\sig(W)= \sig'\tau(W)= \sig'(\ov {\tau(W)})= \ov{\sig'\tau(W)}= \ov{\sig(W)} = \sig(\ov W).$$
By hypothesis, $\alp(a) \neq 1$ for all $a\in B$. 
 Hence, $(\alp,\sig)$ is a \solu at $V$. Since $M(B) = M(B,\es,\theta,\mu) = M(B,\es,\theta',\mu')$, we have $(\alp,\sig') $ is a \solu at $V'$. 
\item[(ii)] The assertion 
  $\alp\sig(W) = \alp \sig'(W')$ is trivial since $W'= \tau(W)$, 
  $\sig=\sig'\tau$.
\item[(iii)] 
For each $X$ write $\tau(X)$ as a word
$$\tau(X)= x_{X,1} \cdots x_{X,\ell_X}$$
with $x_{X,i}\in B\cup \cX'$. Since every $X'\in \cX'$ appears somewhere in 
$\tau(W)$ 
(by \prref{def:wellf2}(5)) 
we obtain:
$\cX' \sse \bigcup\set{x_{X,i}}{X\in \cX \wedge 1\leq i \leq\ell_X}.$
Hence 
\begin{align}
\Abs{\alp,\sig}&= \sum_{X\in \cX}\abs{\alp\sig(X)} = \sum_{X\in \cX}\abs{\alp\sig'\tau(X)} \\
        &= \sum_{X\in \cX}\abs{\alp\sig'(x_{X,1} \cdots x_{X,\ell_X})} 
        = \sum_{X\in \cX, 1 \leq i \leq \ell_X}\abs{\alp\sig'(x_{X,i})}\\
        \label{eq:greater}&\geq \sum_{X'\in \cX'}\abs{\alp\sig'(X')} = \Abs{\alp,\sig'}.
\end{align}
\item[(iv)] If there is some $X$ with 
$\tau(X) \in \cX'^*a\cX'^*$ where $a \in B$ and $\alp(a) \neq 1$, then some
$x_{X,i}= a\notin \cX'$ with $\alp \sig'(a) = \alp(a)\neq 1$. Hence, 
$\abs{\alp \sig'(x_{X,i})}\geq 1$; and the $\geq$ in (\ref{eq:greater}) becomes the inequality $>$. 
\end{itemize}
\end{proof}

In the second lemma we consider the  \morph{s} which leave all variables invariant, and conclude that such a morphism does not change the weight of a solution. 
\begin{lemma}\label{lem:contr}
Let $V= (W,B,\cX,\theta,\mu)$ and $V'= (W',B',\cX,\theta',\mu')$  be \exe{s}, 
$h: M(B',\cX,\theta',\mu') \to M(B,\cX,\theta,\mu)$ be an $(A\cup \cX)$-\morph, and $\alp: M(B) \to  M(A,\es,\emptyset,\mu_{0})$ be an $A$-\morph where $M(B) = M(B,\es,\theta,\mu)$ such that the following conditions are satisfied. 
\begin{itemize}
\item $W = h(W')$.
\item $\alp(a) \neq 1$ for all $a\in B$.
\item If $\cX \neq \es$, then $h(a') \neq 1$ for all $a'\in B'$.
\item If $\theta(X) = c\in B$ for some $X\in \cX$, then $c\in B'$, $\theta'(X) = c$, and $h(c) \in c^*$.
\end{itemize}

Given a $B'$-solution $\sig'$ at $V'$, define a $B$-\morph $\sig: M(B,\cX,\theta,\mu) \to  M(B)$
by $\sig(X)=h\sig'(X)$. 
Then 
$(\alp,\sig)$ is a \solu at $V$ and $(\alp h,\sig')$ is a \solu at $V'$.
Moreover, $\alp\sig(W) = \alp h\sig'(W')$
and $$\Abs{\alp,\sig}=\Abs{\alp h,\sig'}.$$ 
\end{lemma}

\begin{proof} 
By definition, $\mu h = \mu'$ and $\mu_{0}\alp= \mu$. Hence $(\alp h,\sig')$ is a \solu at $V'$. Now,  
$h(X)= X$ for all $X\in \cX$. 
Hence, $\sig(h(X))=\sig(X)=h\sig'(X)$. For $b'\in B'$ we obtain $\sig h(b')=h(b')=h \sig'(b')$ since $\sig'$ and $\sig$ are the identity on $B'$ and $B$ respectively. It follows that $\sig h=h\sig'$ and hence, $\alp\sig(W) = \alp h\sig'(W')$. 
Next, 
$$
\sig(W) = \sig(h(W')) = h (\sig'(W')) = h (\sig'(\ov{W'}))= \sig (h (\ov{W'})) 
= \sig(\ov{h(W')})= \sig(\ov{W}).
$$ 
Moreover, if $X\in \cX$ and $\theta(X)$ is defined, then $\theta(X) =\theta'(X)= c\in B\cap B'$, and $h(c)\in c^*$ by hypothesis. Hence, 
$\sig'(X)\in c^*$ and therefore $\sig(X)=h\sig'(X)\in c^*$, too.  
Thus, $\sig$ is a $B$-\solu at $V$ and, consequently, $(\alp,\sig)$ a \solu at $V$. 
Finally, since $\sig(X)=h\sig'(X)$ we obtain
$$\Abs{\alp,\sig}= \sum_{X\in \cX}\abs{\alp\sig(X)} = \sum_{X\in \cX}\abs{\alp  h \sig'(X)} = 
\Abs{\alp h,\sig'}.$$
\end{proof}

During the process of finding a solution, the parameters $W,B,\cX,\theta,\mu$
change. We describe the possible changes in terms of a directed graph, which will be converted into an NFA.

\subsection{The NFA $\cF$ and the trimmed NFA $\cA$}\label{subsec:nfacA}
We are ready to define the NFA  $\cA$ mentioned in \prref{thm:alice} in the case 
where $\MMA= A^*$ is a free monoid with \invol. 

\subsubsection{States}\label{subsec:infinite}
We start by building an NFA $\cF$ whose states are all the extended equations $(W,B,\cX,\theta,\mu)$ according to \prref{def:wellf2}. We will later obtain $\cA$ by trimming, that is, by removing all states which are not on accepting paths. Thus, the only difference between $\cF$ and $\cA$ is that $\cA$ doesn't have superfluous states.

\begin{lemma}\label{lem:sizeNFAandV}
An extended equation $V=(W,B,\cX,\theta,\mu)$ can be specified using
at most $\Oh(n\log n)$ bits, so $\cF$ has not more than singly exponentially many states.
\end{lemma}
\begin{proof}
We claim that each component of $V$ can be specified using $\Oh(\abs\Gam)=\Oh(n)$ letters from $\Gam$ plus a finite alphabet. Since  $\abs{\Gam}\in\Oh(n)$, we can encode each letter  in $\Gamma$ plus the finite alphabet as a binary number of length at most $\Oh(\log n)$ bits.
 Thus $V$ can be encoded by a binary string of length in $\Oh(n\log n)$.
 It follows that the total number of extended equations is at most $ 2^{\Oh(n \log n)}$.

To establish the claim, notice that  $W\in \Gamma^*$ with $\abs W\leq 204n$, $B \cup \cX \sse \Gam$, $\theta\subset \Gam\times \Gam$ and $\abs \theta \leq \abs{ B \cup \cX}$. Since 
 $\mu: B\cup\cX\ra  N_\F$ and $N_\F$ is finite, $\mu$ can be encoded as a list $\{(c, \mu(c))\mid c\in B \cup \cX\}$, using letters from $\Gam$ plus the finite alphabet $N_\F$. 
  \end{proof}

\paragraph{Initial states.} 
An {\em initial state} is any state of the form $(\Winit,A,\Xinit,\es,\muinit)$,
where 
$$\muinit:(A \cup \Xinit) \to N_\F$$
 is a \morph extending $\mu_0$ such that 
$\muinit(X) \neq 0$ for all $X\in \Xinit$.

If $(\alp,\sig)$ is a \solu of $(\Winit,A,\Xinit,\es,\muinit)$, then necessarily
$\alp = \id{A^*}$ since $\alp$ leaves the letters from $A$ invariant. 
Moreover, we know that $\muinit(X)= \mu_0\sig(X)$. This means that the initial 
value of $\muinit(X)$ tells us whether $\sig(X)=1$; and if $\sig(X)\neq 1$, then 
$\muinit(X) = (a,b)$ and $\sig(X) \in aA^* \cap A^*b$. Hence, $\muinit(X)$ specifies the first and last letters of the reduced word $\sig(X)$ whenever $\sig(X)\neq 1$.
Moreover, $\muinit(X) \neq 0$ implies $\alp \sig(X) \in \F$. Hence, $\alp \sig(X)$ is a reduced word in $\Apone^*$. 

\paragraph{Final states.}
We choose and fix ``distinguished'' letters $c_1 \lds c_m \in C\sm A$ such that $c_i\neq c_j \neq \ov{c_i}$ for all $i\neq j$. 
We say that a state 
$(W,B,\es,\es,\mu)$ is \emph{final} if 
\begin{enumerate}
\item $W= \ov W$,
\item The word $W$ has a prefix of the form $\#c_1\# \cdots \# c_m \#$. 
\end{enumerate}

Every final {state} has the unique $B$-solution $\sig= \id B$ because final states don't have any variables. 

\begin{remark}
The names \emph{initial} and \emph{final} refer to the phase in the construction of the graph at which a state is produced, rather than being start or accept states for the NFA. That is, when we obtain the \edtol language characterization, the start states of the NFA recognising the rational language of endomorphisms correspond to the final states defined here, and the accept states correspond to the initial states. 
\end{remark}

\subsection{Transitions}\label{subsec:traG}
We define two different forms of transitions, based on substitutions and compressions. Both forms are labeled by an \Endo of $C^*$ which induces a \morph between partially commutative monoids $M(B,\es,\theta,\mu)$ and $M(B',\es,\theta',\mu')$.

The direction of each \tra is opposite to that of the morphism labelling the \tra.
Suppose we have a path $p$ from an initial to a final state. A very important (and, perhaps, initially counterintuitive) fact is that in order to produce solutions, our algorithm
follows the path $p$ backwards, that is, from the final to the initial state; we compose the morphisms labeling the transformations in such a directed path $p$ from 
 the last edge to the first one, in order to produce the solutions. This is in agreement with our initial and final states being accept and start states in the NFA, respectively. 

\subsubsection{Substitutions.}\label{subsec:subst}
A \emph{\subst \tra} transforms the variables and does not affect the constants.
Let $V= (W,B,\cX,\theta,\mu)$ and  
$V'= (W',B,\cX',\theta',\mu')$ be {states} in $\cF$
sharing the same set of constants $B$; and assume that $V$ is not final and that $V'$ is not an initial state. Moreover, let 
$\theta(b) = \theta'(b)$, 
and $\mu(b) = \mu'(b)$ for all $b\in B$. Therefore $M(B) = M(B,\es,\theta,\mu)= M(B,\es,\theta',\mu')$.

Let $\tau: M(B,\cX,\theta,\mu) \to M(B,\cX',\theta',\mu')$ be any $B$-\morph such that
$\tau(W)= W'$, $\tau$ modifies only $X$ and $\ov X$ for some variable 
$X$, leaves all 
$x\in (B\cup \cX)\sm \os{X,\ov X}$ invariant, and 
$$\tau(X) \in (B\cup \cX')^* \text{ with }\abs{\tau(X)}\leq 3.$$  
Furthermore, we only allow the following choices for $\tau(X)$, $\cX$ and $\cX'$:
\begin{itemize}
\item[(i)] $\tau(X) = 1$ and $\cX' = \cX \setminus  \os{X,\ov X}$.
\item[(ii)] $\tau(X) = uX$ and $\cX' = \cX$ with $u \in B^*$ and $1 \leq \abs u \leq 2$.
\item[(iii)] $\tau(X) = cX'X$ and $\cX = \cX' \setminus  \os{X',\ov{X'}}$ with $c \in B$ and $\theta'(X')=c$.
\end{itemize}

In each of these three cases we define the substitution \tra:
$$V = (W,B,\cX,\theta,\mu) \arc \eps (\tau(W),B,\cX',\theta',\mu')= V'.$$
Here, the label $\eps$ denotes the identity morphism $\id{C^*}$, it restricts to the identity \morph {} from $M(B,\es,\theta',\mu')$ to $M(B,\es,\theta,\mu)$,
and it will be applied in the opposite direction from $\tau$ and the transition.
Note that after having performed a substitution \tra we have 
$\Abs{V'} < \Abs V$ \IFF $\tau$ is defined by $\tau(X) =1$ for some $X$. %

\subsubsection{Compressions}\label{subsec:trimcA}
A \emph{compression transition} affects the constants, but does not change the variables.
Let $V= (W,B,\cX,\theta,\mu)$ and  
$V'= (W',B',\cX,\theta',\mu')$ be {states} in $\cF$ 
sharing the same set of variables $\cX$ and assume $V$ is not a final state,
$\theta(X) = \theta'(X)$ 
and $\mu(X) = \mu'(X)$ for all $X\in \cX$.

Let $h: M(B',\cX,\theta',\mu') \to M(B,\cX,\theta,\mu)$ be any $(A\cup \cX)$-\morph such that $W= h(W')$ and  
\begin{enumerate}
\item if $V'$ is non-final, then $1 \leq \abs{h(c)} \leq 2$ 
for all $c \in B'$,  
\item if $V'$ is final, then
$\sum_{c\in B'} \abs{h(c)} \leq \abs W$. 
\end{enumerate}

In case that either $\Abs{V} > \Abs {V'}$ or $V'$ is final and $h\neq \id{B^*}$, we define a compression transition in $\cF$ by
$$V = (h(W'),B,\cX,\theta,\mu) \arc h (W',B',\cX,\theta',\mu')= V',$$
where the \tra label $h$ is given by an \Endo $h \in \End(C^*)$ which induces 
the \morph $h: M(B',\cX,\theta',\mu') \to M(B,\cX,\theta,\mu)$ and which leaves all letters not in $B'$ invariant. 
The direction of the \morph $h$ is again opposite to that of the \tra.

\begin{remark}\label{rem:finspec}
The reason that we have to treat \tra{s} to final states differently is twofold. First, the coexistence of ``singular'' and ``nonsingular'' \solu{s} is possible. In the singular case we have $\sig(X)=1$ for some $X$ and in the nonsingular case we have $\sig(X)\neq 1$ for all $X$. Say there are \solu{s} $\sig$ and $\sig'$ such that $\sig(X_1)=1$ and $\sig'(X_1)=a \in \Apone$. Then for some $h,h' \in L(\cA)$ and some $c_1$ we must have $h(c_1) = 1$ and $h'(c_1) = a$. 
Thus in transformations to a final state we must allow that
$h$ maps some letters to the empty word. In all other situations this is forbidden. 
Thus, if $V\arc h V'$ is a compression \tra and $V'$ is final, then we allow $\Abs V < \Abs {V'}$. 

Second, if a state $V=(W,B,\es,\theta,\mu)$ has no variables, then $W$ has prefix $\#u_1\#\cdots \#u_m\#$ with $u_i\in C^*$. In this case we wish to allow a compression transition $h$ to a final state in one step. By imposing the condition $\sum_{c\in B'} \abs{h(c)} \leq \abs W$ we make sure the specification of $h$  fits into our linear space bound, which is crucial in our complexity analysis below.
\end{remark}

\begin{example}
Let $U=aX$ and $V=aaab$ be an equation, for the purposes of demonstrating how the graph or NFA works. We have 
$$\Winit=\#X\#aX\#aab\#\ov X\ov a\# \ov b\ov a\ov a\#\ov X\#.$$
A path from initial to final states in the graph $\cF$ for this equation is shown in \prref{fig:exampleF}, where for simplicity we label states by a  prefix of  $W$ in each extended equation.

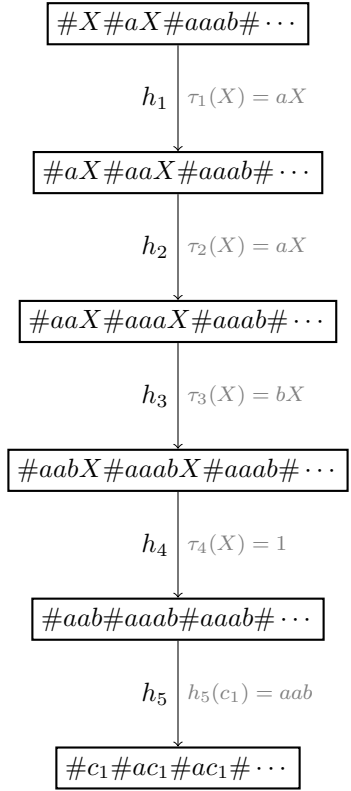
\begin{figure}[htbp]
\begin{center}
  \begin{tikzpicture}[ node distance = 40pt]
    \tikzstyle{ann} = [draw=none,fill=none,right]
  
 \node[draw, thick, rectangle] (a) {$\#X\#aX\#aaab\#\cdots$};
  \node[draw, thick, rectangle] (b) [below=of a] {$\#aX\#aaX\#aaab\#\cdots$};
  \node[draw, thick, draw, thick, rectangle] (c) [ below=of b] {$\#aaX\#aaaX\#aaab\#\cdots$};
  \node[draw, thick, rectangle] (d) [ below=of c] {$\#aabX\#aaabX\#aaab\#\cdots$};  
  \node[draw, thick, rectangle] (e) [ below=of d] {$\#aab\#aaab\#aaab\#\cdots$};  
      \node[draw, thick, rectangle] (f) [ below=of e] {$\#c_1\#ac_1\#ac_1\#\cdots$};
      
    \draw[->](a) --  node[left]  {$h_1$} node [right] {\footnotesize \color{gray} $\tau_1(X)=aX$}  (b);     %
    \draw[->](b) --  node[left] {$h_2$} node [right] {\footnotesize \color{gray} $\tau_2(X)=aX$}(c);
    \draw[->](c) --   node[left] {$h_3$} node [right] {\footnotesize \color{gray} $\tau_3(X)=bX$} (d);     
    \draw[->](d) --   node[left] {$h_4$} node [right] {\footnotesize \color{gray} $\tau_4(X)=1$} (e);
        \draw[->](e) --   node[left] {$h_5$} node [right] {\footnotesize \color{gray} $h_5(c_1)=aab$} (f);
 \end{tikzpicture}
  \caption{A path in $\cF$ from initial to final state for the equation $aX=aaab$. The solution $\sigma(X)$ is obtained by applying the maps $h_1,h_2,h_3, h_4,h_5$ to $c_1$ in reverse order, that is, $\sigma(X)=h_1h_2h_3h_4h_5(c_1)$.
 }\label{fig:exampleF}
\end{center}
\end{figure}

The first four transitions are substitutions $\tau_1(X)=\tau_2(X)=aX,\tau_3(X)=bX,\tau_4(X)=1$ so $h_1,h_2,h_3,h_4$ are just $\id{C^*}$, and the map $h_5(c_1)=aab$ is a compression to a final state.
A solution for $X$ can be obtained by applying the maps to $c_1$ in reverse order to the path labelling,  so we get $\sigma(X)=h_1h_2h_3h_4h_5(c_1)=h_1h_2h_3h_4(aab)=aab$.
\end{example}

\subsection{Proof that the NFA is constructed in quasi-linear space}\label{subsec:trimcA}

We can now give the algorithm to construct the trim NFA $\cA$  in $\NSPACE(n\log n)$. We first give an algorithm to construct $\cF$, then use this to construct $\cA$.

\begin{lemma}\label{lem:checkstate}
Given a tuple $V= (W,B,\cX,\theta,\mu)$, where $W\in \Gam^*$, 
$B \sse C$, $\cX \sse \OO$, $\theta$ is a type, and $\mu:(B\cup \cX) \to N$ is a mapping, we can check within
$\NSPACE(n\log n)$ whether $V$ is an extended equation (that is, $V$ is  a state in $\cF$) and furthermore decide whether the state $V$ is initial or final. 
\end{lemma}

\begin{proof} 
As noted in \prref{lem:sizeNFAandV},  writing down any extended equation requires at most $\Oh(n\log n)$ bits, so if $V$ requires more space we reject it as a valid input. 
If $V$ fits into the allowed space, then go through the conditions listed in \prref{def:wellf2}.
It is obvious how to check the first five conditions.  For example, 
if $\abs W > 204n$, then we reject immediately. 

The most involved test is to see that for every factor $u$ of every $u_i$ with the interpretation $u_i \in M(\Gam,\theta)$ the element $\ov u$ also appears
 in $W\in M(\Gam,\theta)$. For this test we invoke the algorithm that solves the uniform factor problem in free partially commutative monoids as explained in 
 \prref{subsec:fpcm}. Recall that the uniform factor problem refers to an input of the  
 form $(\Gam,\theta,u,w)$. In our case the input has the specific form 
 $(\Gam,\theta,\ov{u},W)$.  We presented a nondeterministic algorithm using linear space in the input size, where the input size of a tuple $(\Gam,\theta,u,w)$ is 
 $(\abs \Gam + \abs \theta + \abs {uw})\log \abs \Gam$,  as we need $\Oh(\log \abs \Gam)$ bits to encode letters. Since $(\abs \Gam + \abs \theta + \abs {uw})\log \abs \Gam \in \Oh(n\log n)$, the call of such a subroutine fits into our space bound. 
 
 Having completed the check that $V$ is a state of $\cF$, it is easy to check whether it is initial ($W=\Winit$, $B=A$, $\theta=\emptyset$) or final ($W=\ov W$, $\theta=\emptyset$, $\cX=\emptyset$); since $\theta=\es$ in both cases we are just checking $W=\Winit, W=\ov W$ in a free monoid.
\end{proof}

In the following, when we say that $V= (W,B,\cX,\theta,\mu)$ is a state in $\cF$, this means $V$ is given as a tuple for which the syntax check according to \prref{lem:checkstate} that $V$ is indeed a state was performed.

\begin{lemma}\label{lem:checktra}
Given states $V= (W,B,\cX,\theta,\mu)$, $V'= (W',B',\cX',\theta',\mu')$ in $\cF$, and a mapping $h:B' \to B^*$, we can check  within
$\NSPACE(n\log n)$ whether the triple $(V,V',h)$ encodes an 
\tra $V\arc h V'$ in the graph $\cF$. 
\end{lemma}

\begin{proof} 
We assume $h$ is specified as a tuple requiring at most $\Oh(n\log n)$ bits.
In order to check whether $V\arc h V'$ is
a compression \tra we must have $h\neq\id{B^*}$ and then we go through the conditions of \prref{subsec:trimcA}, most of which are immediate to verify.  Among these, we have to compute  $h(W')$ as a word in 
$(B\cup \cX)^*$ and then see if $W = h(W') \in M(B\cup \cX,\theta)$. The test $W = h(W')\in M(B\cup \cX,\theta)$ is a special case of the uniform factor problem in free partially commutative monoids, as already discussed in the proof of \prref{lem:checkstate}.

For a \subst \tra, a necessary condition is $B= B'$ and $h= \id{B}$, which is trivial to check. Next we 
guess some mapping  $\tau: \cX \to (B\cup \cX')^*$ 
with $\abs{\tau(X)} \leq 3$ for all $X \in \cX$. Just as above we check $\tau(W) = W' \in M(B'\cup \cX',\theta')$ and the other requirements  
for \subst{s} listed in \prref{subsec:subst}.
\end{proof}

As usual in automata theory we modify the NFA $\cF$ by removing all states which are not on a path from some initial to some final state. 
If there is no such path, then $L(\cF)$ is the empty set. The resulting NFA will be denoted as $\cA$. We have $L(\cA) = L(\cF)$. Moreover, $L(\cA) = \es$ \IFF the automaton $\cA$ is empty. 

The key tool used to build the trim NFA $\cA$ is $\ispath{V}{V'}$, which we define to be a Boolean predicate that yields \emph{true} \IFF there is a path 
from state $V$ to $V'$ in the graph $\cA$.

\begin{lemma}\label{lem:checkpath}
Let $V, V'$ represent two states in the graph $\cF$. Then the predicate $\ispath{V}{V'}$ can be evaluated in $\NSPACE(n\log n)$.
\end{lemma}

\begin{proof} 
Define the language $L_\cF=\{(V,V')\mid \ispath{V}{V'}=\text{true}\}$. 
On input $(V,V')$ we can guess a path $V=V_0, V_1, h_1, V_2, h_2,\cdots, V'=V_k,h_k$ in $\cF$ from $V$ to $V'$ and check for each $i$ whether $(V_{i-1},V_i, h_i)$ encodes a \tra by using
Lemmas~\ref{lem:checkstate} and \ref{lem:checktra}. Thus, $L_\cF\in\NSPACE(n\log n)$.

Since $\NSPACE(n \log n )$ is closed under complementation by Immerman and Szelepcs{\'e}nyi (see \cite[Theorem 7.6]{pap94}), we also have
 \[\ov{L_\cF}=\{(V,V')\mid \text{$\not\exists$  a path from $V$ to $V'$ in $\cF$}\}\in\NSPACE(n\log n).\]
Thus, the predicate $\ispath{V}{V'}$ can be evaluated in $\NSPACE(n\log n)$ by running two procedures simultaneously to determine if $(V,V')\in L_\cF$ or 
$(V,V')\in\overline{L_\cF}$.
\end{proof}

\begin{proposition}\label{prop:contrim}
We can construct  the trim NFA $\cA$ in $\NSPACE(n \log n)$. 
Within  the same space complexity we can decide  whether $\cA$ is empty, or whether $\cA$ contains a directed cycle. 
\end{proposition}

\begin{proof}
For each $V$ that is a state of $\cF$ output $V$ as an initial node 
of  $\cA$ if both (1) $V$ is initial in $\cF$, and (2) there exists some 
path to a final state in $\cF$. We check (1) using Lemma \ref{lem:checkstate}. For (2) we run through 
 all final states $V'$ of $\cF$ and evaluate 
the predicate $\ispath{V}{V'}$. If at some point $\ispath{V}{V'}$ becomes true, we output $V$ as an initial node in $\cA$. 
If no initial node in $\cA$ is found, then we stop; the output is $\cA = \es$. Hence, we continue only if
there is at least one initial node.

Next, we construct all transitions of $\cA$ as follows. We list all triples $(V,V',h)$ where $V \arc h V'$ is a \tra in $\cF$.
For each such triple we consider all states $V_0$ of $\cA$ which are initial, and
for each $V_0$ we evaluate $\ispath{V_0}{V}$. If no such $V_0$ is found where 
$\ispath{V_0}{V}$ is true, then we 
move to the next triple $(V,V',h)$. If at least one such $V_0$ exists, we list 
all states $V_f$ of $\cF$ which are final. 
For each $V_f$ we evaluate $\ispath{V'}{V_f}$. If no such $V_f$ is found where 
$\ispath{V'}{V_f}$ is true, then we 
move to the next triple $(V,V',h)$. Otherwise we output $(V,V',h)$ as a \tra 
of $\cA$. If, moreover, $V'$ is final in $\cF$, then we mark that \tra in order to indicate that $V'$ is final in $\cA$, too. We then
move to the next triple $(V,V',h)$.

Having these two lists at hand we 
have constructed the trim NFA $\cA$.

Finally, to check for a directed cycle we enumerate all pairs $(V,V') \in \cA \times \cA$ with $V \neq V'$ and for each pair evaluate
$\ispath{V}{V'}$ and $\ispath{V'}{V}$.\end{proof}

With the assertion in \prref{prop:contrim} 
the algorithmic part of the proof of the monoid version of \prref{thm:alice} is finished.  
It remains to show the soundness and completeness 
of the construction. This requires purely existential statements, where no reference to effectiveness is necessary.

\subsection{Soundness}\label{subsec:sound}
In this section we prove \emph{soundness},  that is, any output we obtain by following the transitions in the NFA $\cA$ from an initial to a final state, and then applying the corresponding maps in reverse  order to the distinguished letters, gives a correct solution to the equation $\Winit$.

Recall that we have chosen distinguished letters $c_{1}\lds c_{m}\in C$, and that if $(W,B,\es,\es,\mu)$ is a final state, then 
$W = \ov W$ and $W\in  \#c_{1}\#\cdots \#c_{m}\#B^*$.
\begin{proposition}\label{prop:backandforth}
Let $V_{0}\arc{h_1} 
\cdots \arc{h_{t}} V_t$
be a path in $\cA$ of length $t$, where $V_0= (\Winit, A, \Xinit,\es ,\muinit)$ 
is  an initial and $V_t=(W,B,\es,\es,\mu)$ is a final {state}. 
Then $V_{0}$ has a solution $(\id{A^*},\sig)$ with
 $\sig(\Winit)= h_1 \cdots h_t(W)$. 
Moreover, for $1 \leq i \leq m$ we have
  $$ 
\sig(X_{i}) = h_1 \cdots h_t(c_{i}).
$$ 
\end{proposition}

\begin{proof}
Let $s\geq 0$ and  $V_{0}\arc{h_1} 
\cdots \arc{h_{s}} V_s$ be any path to some state 
$V_s=(W_s,B,\cX,\theta,\mu)$ such that $\sig_s$ is a $B$-\solu at $V_s$.
We claim that $V_{0}$ and $V_{s}$ have solutions $(\id{A^*},\sig)$ 
and $(\id{A^*}h_1 \cdots h_s,\sig_s)$, respectively,
with
\begin{align}\label{eq:wonderc}
\sig(\Winit) = h_1 \cdots h_s\sig_s(W_s). 
\end{align}  
Claim (\ref{eq:wonderc}) is trivial for $s=0$ and for 
$s>0$ it follows by induction using \prref{lem:contr}
or \prref{lem:varsub}, depending on whether $h_s$ is a \subst \tra or a compression \tra. Now for $s=t$ we have
 $\ov{W} = {W}$ by the definition of a final state. Since no variables occur in $W$, $\sig_t = \id {B^*}$ is the (unique) $B$-\solu of $W$, so $\sig(\Winit)= h_1 \cdots h_t(W)$. 
 
By definition $\#X_{1}\#\cdots \#X_{m}\#$ is a prefix of $\Winit$ and $\#c_{1}\#\cdots \#c_{m}\#$ is a prefix of $W$ for the final state $V_t$, but ${h} = \id{A^*}h_1 \cdots h_t$ is 
an $A$-\morph{} from $B^*$ to $A^*$ with 
$\abs{{h}(c)}_\# =0$
for all $c \in B$. This implies 
$$\sig(\#X_{1}\#\cdots \#X_{m}\#)= {h}(\#c_{1}\#\cdots \#c_{m}\#).$$
In particular, $\sig(X_{i}) = h_1 \cdots h_t(c_{i})$ for $1 \leq i \leq m$.
\end{proof}
  
Using the notation of \prref{thm:alice} we have shown 
soundness, that is, every output we obtain is a solution in reduced words. 

\begin{corollary}\label{cor:sound}
The following inclusion holds:
\begin{align*}
\{(h(c_1) \lds  h(c_m))&\in C^*\times \cdots \times C^*\mid h\in L(\cA)\}
\sse \\
\bigcup_{\set{\mu}{\mu(X) \neq 0}} \{(\sig(X_1) \lds  \sig(X_m))&
 \in \F^m 
 \mid \sig \in \Sig \wedge \sig(\Winit)=\ov{\sig(\Winit)} \wedge \mu = \mu_0\sig\},
 \end{align*}
 where $\Sig$ denotes the set of $C$-\morph{s} $\sig: \Gam^* \to C^*$.
\end{corollary}

\begin{proof}
Follows from \prref{prop:backandforth}.
\end{proof}
\begin{corollary}\label{cor:infsol}
If the NFA $\cA$ is nonempty, then there is some solution $\sig$ which maps all variables 
$X_i$ to reduced words in $\Apone^*$ and which satisfies 
$\sig(\Winit)=\ov{\sig(\Winit)}$.

If the NFA $\cA$ contains a directed cycle, then there are infinitely many such  $\sig$.
\end{corollary}
\begin{proof}
 The first part follows from \prref{prop:backandforth}. 
 
 Now assume that  $\cA$ contains a directed cycle. Then for every $t_0\in \N$ we can choose
 a path $V_{0}\arc{h_1} 
\cdots \arc{h_{t}} V_t$ from an initial state $V_0$ to some final state 
$V_t$ with $t > t_0$. For each 
$0\leq s \leq t$ 
define $\alp_s = \id{A^*}h_1 \cdots h_s$. Thus, $\alp_0 = \id{A^*}$. We  view $\alp_s \in \End(C^*)$, and let $(\alp_s,\sig_s)$ be the corresponding \solu at $V_s$, which exists due to (\ref{eq:wonderc}).

For every \tra $V_{i-1}\arc{h_{i}} V_i$ 
which is defined either by a compression, or by a \subst of type (i), we have $\Abs{V_{i-1}} >\Abs{V_i}$. Since $\Abs V \in \Oh(n^4)$ for all states, 
there is a  constant $\kappa'$ such that every path of length $\kappa' n^4$ must include a \subst of type (ii) or (iii).
Thus, we may assume that for a large enough $t$
there are more than $t_0$ \tra{s} where $V_{i-1}\arc{h_{i}} V_i$ is defined by a \subst of type (ii) or (iii), i.e. with $\tau(X) \in \Gam^*C\Gam^*$. 

By the definition of $\cA$ we have $\alp_s(c) \neq 1$ for all $c \in C$ whenever $s<t$. (The final transition is an exception.) By \prref{lem:varsub} 
and \prref{lem:contr} we have
$$\Abs{\alp_0,\sig_0} \geq t_0.$$
since for each compression transition the weight is unchanged, and for each substitution the weight decreases, and in particular, it decreases strictly at least $t_0$ times. 
The result follows since $\alp_0= \id{A^*}$. Hence, there infinitely many 
\solu{s} $\sig_0$. 
\end{proof}

\subsection{Completeness}\label{subsec:compness}
Now we show that every solution of the equation $\Winit$ can be obtained from $\cA$.

Let us  fix some state 
$V = (W, B, \cX,\es ,\mu)$ and assume that $V$ 
has a \solu  $(\alp,\sig)$. 
We will show that if $V$ is ``small enough'', then  $\cA$ contains a 
path $V\arc{h_1} V_1 
\cdots \arc{h_{t}} V_t$ to some final {state} 
$V_t=(W',B',\es,\es,\mu')$ such that 
$\sig(W) = h_{1}\cdots h_{t}(W')$. 
Let us make precise what ``small'' means.
\begin{definition}\label{def:small}
A state $V = (W, B, \cX,\es,\mu)$  is called \emph{small} if 
$$|W| \leq \bcc.$$ 
\end{definition}
Clearly every initial state
is small. Final states  need not be small. 
\subsubsection{Forward property of transitions}\label{subsec:fpotransitions}
The  existence 
of a path $V\arc{h_1} V_1 
\cdots \arc{h_{t}} V_t$ to some final state 
$V_t=(W',B',\es,\es,\mu')$ such that 
$ 
\sig(W) = h_{1}\cdots h_{t}(W')
$ 
relies on the following technical concept. 
\begin{definition}\label{def:fp}
Let $V=(W,B,\cX,\theta,\mu) \arc h (W',B',\cX',\theta',\mu')= V'$ be a transition in $\cA$ and $(\alp,\sig)$ be a \solu at $V$.
We say that the triple
$(V\arc h V',\alp,\sig)$ satisfies the {\em forward property} 
if there exists a \solu $(\alp h,\sig')$ at $V'$ such that 
$$ \alp \sig(W) = \alp h \sig'(W').$$
\end{definition}
By a slight abuse of language: if $V\arc h V'$ is a \tra in $\cA$ and the solution 
$(\alp,\sig)$ at the source $V$ is clear from the context, then we say also  that the 
\tra  $V\arc h V'$ satisfies the forward property. In particular, if we follow a path from $V$ having a \solu $(\alp,\sig)$ to some state 
$V'=(W',B',\es,\theta',\mu')$ by transitions satisfying the forward property, then $V'$ has some \solu. But as $V'$ uses no variables, we obtain $W' = \ov{W'}$. 

\begin{lemma}\label{lem:ftaus}
Let $V= (W,B,\cX,\theta,\mu) \arc \eps (\tau(W),B,\cX',\theta',\mu')=V'$ be 
a \subst \tra (according to \prref{subsec:subst}) and $\theta(Y) = \theta'(Y)$ for all $Y \in \cX \cap \cX'$. In each of the 
following cases $(V\arc \eps V',\alp,\sig)$ satisfies the forward property:
\begin{enumerate}
\item $\sig(X)=1$ and the \tra $V\arc \eps V'$ removes $X$ by
$\tau(X)= 1$;
\item $\theta= \es$, $\sig(X) = av$, $\mu'(X) = \mu(v)$, and the \tra $V\arc \eps V'$ is defined by $\tau(X)= aX$;
\item $\theta(X)=\es$, $\sig(X) = cuv$, $u \in c^*$, $\mu'(X') = \mu(u)$, $\mu'(X) = \mu(v)$, and the \tra $V\arc \eps V'$ is defined by $\tau(X)= cX'X$ with $\theta'(X') = c$;
\item $\theta(X) = c$, $\sig(X)= cu$, $\mu'(X) = \mu(u)$, and the \tra $V\arc \eps V'$ substitutes $X$ by $\tau(X)= cX$.
\end{enumerate}
\end{lemma}

\begin{proof}Let $V\arc \eps V'$ be defined by $\tau: M(B,\cX,\theta,\mu) \to M(B,\cX',\theta',\mu')$. It is enough to show that $V'$ has a $B$-\solu 
with $\sig= \sig'\tau$.  
\begin{enumerate}
\item Let $\sig'$ be the restriction of $\sig$ to $\cX'= \cX\sm\os{X,\ov X}$. 
Then we have $\sig= \sig'\tau$. 
\item 
Recall that by definition of a \subst \tra{s}, we have $\theta' = \es$, too.
Define $\sig'$ by $\sig'(X) = v$ and $\sig'(Y) = \sig(Y)$ for $Y\neq X,\ov X$. 
Since $\mu'(X)= \mu(v)$, we obtain $\sig'$ as a \morph; and 
we have $\sig= \sig'\tau$.
\item Define $\sig'(X') = u$, $\sig'(X)= v$ and $\sig'(Y) = \sig(Y)$ for $Y\neq X',\ov {X'},X,\ov X$. Then we have $\sig= \sig'\tau$.
\item 
Define $\sig'(X)= u$ and $\sig'(Y) = \sig(Y)$ for $Y\neq X,\ov X$. Since $\theta(X) = c$ and $\sig$ is a \solu, we have $u\in c^*$ and as $\tau$ is a \morph we have $\theta'(X) = c$, too. 
Then we have $\sig= \sig'\tau$.
\end{enumerate}
In all cases it is clear that $\sig'$ is a $B$-\solu.
\end{proof}

\begin{lemma}\label{lem:compitras}
Let $B' \subseteq B$ and $V= (h(W'),B,\cX,\theta,\mu) \arc h (W',B',\cX,\theta',\mu')=V'$ be a compression \tra  (according to \prref{subsec:trimcA}).
If $\sig: \cX \to M(B,\es,\theta,\mu)$ factors through \morph{s} as 
$$\sig: \cX \arc {\sig'} M(B',\es,\theta',\mu') \arc h M(B,\es,\theta,\mu)$$
such that $\sig'(X) \in c^*$ whenever $\theta'(X) = c$,
 then $(\alp h,\sig')$ is a \solu at $V'$ and $(V\arc h V', \alp, \sig)$ satisfies the forward property.
\end{lemma}
\begin{proof}
 We have $\sig h = h \sig'$ and hence, $ \alp \sig(W) = \alp h \sig'(W')$. 
\end{proof}

Frequently, we cannot apply \prref{lem:compitras} because $\sig$ cannot be written as $h\sig'$. The typical example is that  
$B'\varsubsetneq  B$, but some $\sig(X)$ uses a letter from $B\sm B'$, and $h(a) = a$ for all $a \in B'$.
 This type of 
``alphabet reduction'', switching from a larger alphabet $B$ to some proper subset $B'$, is needed only if the type relations $\theta, \theta'$ are empty. 
Therefore the following lemma applies in this situation. 
\begin{lemma}\label{lem:alpred}
Let $B'\varsubsetneq  B$ and $V= (W,B,\cX,\es,\mu) \arc \eps (W',B',\cX,\es,\mu')=V'$ be a compression \tra which is induced by  the identity $\id{C^*}$.
Thus, $\eps$ becomes the canonical inclusion of $M(B',\es,\es,\mu')$ into $M(B,\es,\es,\mu)$. 
In particular, $W= W'$ and $\mu'$ is the restriction of $\mu$. 

Let $(\alp,\sig)$ be a \solu at $V$.
Define a $B'$-\morph $\bet: M(B,\es,\es,\mu) \to M(B',\es,\es,\mu')$ 
by $\bet(b)= \alp(b)$ for $b \in B\sm B'$ and $\bet(b)= b$ for $b \in B'$.
Let 
$\sig'(X) = \bet \sig (X)$.
 Then $(\alp \eps,\sig')$ is a \solu at $V'$ with $ \alp \sig(W) = \alp \eps \sig'(W')$. In particular, $(V\arc \eps V', \alp, \sig)$ satisfies the forward property. 
\end{lemma}

\begin{proof}
Since $\alp:  M(B,\es,\es,\mu) \to M(A,\es,\es,\mu_{0})$ 
is an $A$-\morph with $\mu(a) = \mu_0(a)$ for all $a \in A$, we 
have $\mu\bet(b)= \mu \alp(b) = \mu_{0}\alp(b) = \mu(b)$ for all $b\in B\sm B'$ and $\bet$ is indeed a $B'$-\morph{} from $M(B,\es,\es,\mu)$  to $M(B',\es,\es,\mu')$. 

Note that $M(B',\cX,\es,\mu')$ is a submonoid of $M(B,\cX,\es,\mu)$
and $\eps$ realizes the inclusion of these free monoids. Hence
$W= \eps(W') =W'$ as words. In particular, $\sig(W) = \sig ({\ov W})$ implies $\sig'(W') = \sig' ({\ov W'})$. Thus, $(\alp \eps, \sig')$ solves 
$V'$. 

Finally, by definition of $\bet$ we have $\alp = \alp \bet$ because $\alp$ is an $A$-\morph.  Hence $\alp = \alp \eps \bet$ and we obtain $$\alp \eps \sig'(W')= \alp \eps \sig'(W)= \alp \eps \bet \sig(W)
= \alp  \sig(W).$$
\end{proof}

\begin{definition}\label{def:visible}
Let $\sig: \Gam \to C^*$ be any $C$-\morph and  $W \in\Gam^*$.
The word $W$ is realized as a sequence of {\em positions}, say $1,2\lds \abs{W}$, and each position is labeled by a letter from $\Gam$. 
If $W= u_{0}x_{1}u_{1} \cdots x_{m}u_{m}$, with $u_{i}\in C^*$ and 
$x_{i}\in \OO$, then we have $\sig(W)= u_{0}\sig(x_{1})u_{1} \cdots \sig(x_{m})u_{m}$. The positions in $\sig(W)$ corresponding to the positions of the $u_{i}$'s are henceforth called {\em visible}. 
\end{definition}
Given $w=\sig(W)$, each visible position in $w$ can be uniquely identified with a position in $W$, both positions having the same label in $C$. 
Following a path satisfying the forward property makes the 
length of the equation oscillate.
In particular, thoughout the {\em compression method} below the algorithm progresses from small state to small state, but in between the states are not necessarily small.

Proposition \ref{prop:completeness} shows that every solution can be found by tracing a path in $\cA$. 

\begin{proposition}\label{prop:completeness}
Let $V = (W, B, \cX,\es ,\mu)$ be small and let 
$(\alp,\sig)$ be a solution at $V$. 
Then  $\cA$ contains a 
path $V\arc{h_1} V_1 
\cdots \arc{h_{t}} V_t$ to some  final {state} 
$V_t$ 
of transitions satisfying the forward property.

In particular, if $V$ is an initial state, then we have 
$\sig(X_i) = h_{1}\cdots h_{t}(c_i)$ for all 
$1\leq i \leq m$, where $c_1 \lds c_m$ are the distinguished letters.
\end{proposition}

\subsubsection{Reduction of Proposition~\ref{prop:completeness} to Lemma~\ref{lem:complett}}\label{subsec:popco}
 As a base case we let $\cX= \es$: thus, 
$V=(W,B,\es,\es,\mu)$. If $V$ is final, then there is nothing to do. Otherwise, by definition of an \exe, we have
$W \in \#B^*\#$ and $\abs{W}_\# = \abs{\Winit}_\#$. 
Since $\cX = \es$, we have  $(\alp,\sig)= (\alp,\id{B^*})$ and we can write 
$$W = \#u_1\# \cdots \#u_m\# u_{m+1}\#u_{m+2}\#\ov{u_{m+2}}\# \ov{u_{m+1}} \#\ov{u_m}\# \cdots \#\ov{u_1} \#.$$

Define $B_1= A \cup \os{c_1, \ov{c_1} \lds c_{m+2}, \ov{c_{m+2}} }$ 
as a disjoint union where $c_1 \lds c_{m}$ are the distinguished letters. 
Define 
$V_1= (W_1, B_1, \es,\es ,\mu_1)$ with 
$$W_1 = \#c_1\# \cdots \#c_m\# c_{m+1}\#c_{m+2}\#\ov{c_{m+2}}\# \ov{c_{m+1}} \#\ov{c_m}\# \cdots \#\ov{c_1} \#.$$ 
Defining $\mu_1(c_i) =\mu(u_i)$ and 
$h_1(c_i) = u_i$ yields the desired result. 
Clearly, $(\alp h_1,\id{B_1^*})$ is a \solu at the final state 
$V_1$ and the compression \tra $V\arc{h} V_1$ satisfies the forward property. (Note that we could have some $u_i=1$,  so this is where the case distinction discussed in \prref{rem:finspec} is needed.)

The proof of \prref{prop:completeness} 
is by induction on the weight $\Abs{\alp,\sig,V}$. 
It covers the rest of this section.
Throughout the proof, all transitions satisfy the forward property by \prref{lem:ftaus}, \prref{lem:compitras}, and \prref{lem:alpred}; therefore, if we 
know that
$V_i=(W_i,B_i,\cX_i,\theta_i,\mu_i)$ has a $B_i$-solution $\sig_i$ for all $1 \leq i\leq s$, where $s$ is some positive integer, then
we obtain $\sig(W) = h_{1}\cdots h_{s}\sig_s(W_s)$ by \prref{def:fp}.\\

\noindent{\bf Preprocessing.} \label{page:preprocessing}
By the base case we may henceforth assume that  $\cX\neq \es$. If we have $\sig(X) =1$ for some variable, then we follow a \subst \tra removing the variable; and we are done by induction on the weight. 

Thus, without restriction, we can assume $\sig(X) \neq 1$
for all variables. For each $X\in \cX$, if $\sig(X) \in a B^*$ we follow a \subst \tra  defined by $\tau(X) = aX$.  This has the effect of {\em popping out} constants at the start and end of each variable, since each $X$ comes with its involution $\ov X$. Since $W$ has at most $4n$ variables present,  the length of $W$ increases by at most $8n$ 
 and the weight $\Abs{\alp\sig}$ decreases.
 In case that this \subst leads to a situation where a solution maps $X$ to the empty word, we remove $X$ and $\ov X$.
After that we are done by induction on the weight (since $\Abs{\alp\sig}$ is the dominant term in the lexicographic ordering), \emph{unless} we end with $\abs{\tau(W)} > \bcc$, that is, the new state is not small. 
In that case we will have  $\bcc < \abs{\tau(W)} \leq \pcc$. 
Thus, in proving a more general statement, we will not assume that $V$ is small, but that 
$$\bcc < \abs W \leq \pcc.$$

So far, we did not discuss the size of $B$. 
Assume that we are in the situation of \prref{lem:alpred}: there is $B'$ with $A\sse B'\varsubsetneq  B$ 
such that $W\in (B'\cup\cX)^*$, then we can use 
\prref{lem:alpred}; and we are done by induction on the weight. Thus, after preprocessing  we may assume that all letters in $B\sm A$ appear in $W$, that is, $\abs{W}_b\geq 1$ for all $b\in B\sm A$.

During the preprocessing we decreased the weight, but at the end of this phase $V$ may no longer be  small. Therefore, the proof of \prref{prop:completeness} reduces to showing the following lemma.

\begin{lemma}\label{lem:complett}
Let $V = (W, B, \cX,\es,\mu)$ be a state with a \solu  $(\alp,\sig)$ such that 
$\cX \neq \es$ and $\abs {W} \leq \pcc$.
Then $\cA$ contains a  path 
of transitions satisfying the forward property to 
some small state $V'=  (W', B', \cX',\es,\mu')$ with a \solu  $(\alp',\sig')$
such that $\Abs{\alp,\sig,V} \geq \Abs{\alp',\sig',V'}$. \end{lemma}

\subsubsection{Proof of Lemma~\ref{lem:complett}}\label{subsec:polemco}
The assertion of the lemma is trivial, if $V$ is small. That is: 
$\abs W \leq \bcc$. Hence, we may assume $\bcc <\abs {W} \leq \pcc$.  
Let $V = (W, B, \cX,\es,\mu)$ be a state with a fixed \solu  $(\alp,\sig)$ satisfying the hypothesis of \prref{lem:complett}. 
 We describe a
way to find a path through $\cA$ in terms of a procedure which ``knows'' the \solu $(\alp,\sig)$. 

\subsubsection*{Block compression}

We employ block compression only if $W$ contains a factor $b^2$, where $ b\in B$ and $b \neq \#$. Otherwise we move straight to the next procedure, called \emph{pair compression}. 
During the procedure we will increase the length of $W$ by  $\Oh(n)$, but at the end we will arrive at an equation where $\abs{W'} \leq \abs{W}$; and importantly,
$W'$ will not contain any proper factor $b^2$ with $b\in B$ and $b \neq \#$.
We give an example of this procedure in \prref{sec:blockexample}.

\begin{remark}
While this procedure is technical, the idea is quite simple. The goal is to eliminate long blocks $b^\ell$ that are visible in the equation. To do so we use transitions which replace $bb$ by $b$, just two letters at a time. Before we can apply such a compression, we must ensure the length of any maximal block $b^\ell$ with at least part of the block visible must be {\em even}. So first we follow various substitution and compression transitions to arrange this.  
\end{remark}

\begin{enumerate}
\item \textbf{Recording the constants with large exponents}. Due to the previous substitutions $X\mapsto bX$ in the preprocessing step,  we have that for each $X$ if $bX\leq W$ and $b'X\leq W$ are factors with $b,b'\in B$, then $\#\neq b=b'$. 
For each $b\in B\sm \os \#$ 
define two sets:
\begin{align*}
\Lambda_b &=\set{\lam \geq 2}{\exists db^\lambda e \leq \sig(W): d\neq b\neq e \textrm{ and some $b$ in  $db^\lambda e$  is visible}}, \\ 
\cX_b&=\set{X\in\cX}{bX \leq W \wedge \sig(X)\in bB^*}.
\end{align*}

Note that
\begin{equation}\label{eq:smallC}
\sum_{b} \abs{\Lambda_b} + \abs{\cX_b} 
\leq \abs W.
\end{equation}
By \prref{def:wellf2} we have $\Lambda_b = \Lam_{\ov b}$. 
Another fact is crucial: it might be that 
there are $X\in \cX \sm \cX_b$ with $\sig(X)\in bB^*$,
but then to the  left of every occurrence of $X$ there is (the same) letter $b' \in B\sm \os{\#,b,\ov b}$. In this case the block compression procedure does not touch the variable $X$ (although it may change $\sig(X)$).
If, on the other hand, $X\in \cX_b$, then a factor $bb$ crosses the left border for every occurrence of $X$. The first $b$ in such a factor is visible in $W$, the second one is not.

\item \textbf{Introducing the type and renaming of some constants}.
For each $b \in B$ with  $\Lambda_b\neq \es$ we introduce a \emph{fresh} letter $c_{b} \in C \setminus B$ with $\mu(c_{b})= \mu(b)$. In addition, for each 
$\lam \in \Lambda_b$ introduce a fresh letter  $c_{\lam,b}$ with $\mu(c_{\lam,b})= \mu(b)$. The fresh letters are chosen such that $\ov{c_{b}} = c_{\,\ov b}$ and  $\ov{c_{\lam,b}} =  c_{\lam,\ov b}$. Note that $c_{\lam,b}$ and $c_{b}$ are just names for formal symbols realized by fresh letters in the fixed extended alphabet $C$.

We let $B'= B\cup \bigcup\set{c_b,\ov{c}_b,c_{\lam,b},\ov{c}_{\lam,b}}{\lam \in \Lam_b \wedge b \in B}$ and we introduce a type 
by $\theta(c_{\lam,b}) =c_{b}$ for all $\lam \in \Lam_{b}$.
This yields a free partially commutative monoid
$M(B',\cX,\theta,\mu)$. We define an $\cX$-\morph $$h: M(B',\cX,\theta,\mu)\to M(B,\cX,\es,\mu)$$ by  
$h(c_{\lam,b}) = h(c_{b}) = b$. Next, we modify $W$:  
in every factor $db^\lam e$ of $\sig(W)$ with $d\neq b \neq e$ and $\lam \in \Lam_b$ we replace that factor by $dc_b^\lam e$. This defines a new word 
$W'$ such that $h(W') = W$. Note that so far, no $c_{\lam,b}$ does appear in 
$W'$. Let $V' = (W',B',\cX,\theta,\mu)$. Then $V'$ is a state and we can follow the  \tra $V\arc h V'$. We have $\Abs{V'} < \Abs{V}$ since 
$\theta\neq \es$ and this term appears before the number of constants in the weight of a state. (It might be that all $b$ are gone, so we cannot make sure that the second component in the weight decreased.) 
Note that for each $\lam \in \Lam$ at least one position labeled by $c_{b}$ is visible in $W$. 

We rename $V' = (W',B',\cX,\theta,\mu)$ as $V = (W,B,\cX,\theta,\mu)$ and rename the \solu as $(\alp,\sig)$. 

\item \textbf{Splitting the variables starting with special constants}. 
We skip this step if $\cX_b = \es$ for all $b$. Otherwise, 
for each $b \in B$ and $X \in \cX_b$ we write 
$\sig(X) = c_b^\ell w$ for some  $\ell \geq 1$ with $w \notin \os{b,c_b}B^*$.  
We split the variable $X$ by defining  
$\tau(X) = c_bX'X$ where $X' = X'_{b,X} \in \OO \sm \cX$ is a fresh variable, which is assigned a type $\theta'(X')= c_b$. 
Moreover, we let $\mu'(X') = \mu(c_b)^{\ell-1}$, $\mu'(X)=\mu(w)$, $\sig'(X') = c_b^{\ell-1}$,
and $\sig'(X) = w$. The new set of variables is a disjoint union
$$\cX'= \cX \cup \set{X'_{b,X},\ov{X'_{b,X}}}{b \in B \wedge X \in \cX_b}.$$
We obtain a new state  $V'= (\tau(W),B,\cX',\theta',\mu')$
 and a \morph
$$\tau: M(B,\cX,\theta,\mu) \to M(B,\cX',\theta',\mu').$$ 
The \morph $\tau$ defines a \subst \tra $V \arc{\eps}V'$ which pops a letter.
The new \solu at $V'$ is $(\alp, \sig')$. 

We rename  $V'= (\tau(W),B,\cX',\theta',\mu')$ as $V = (W,B,\cX,\theta,\mu)$
and rename the \solu as $(\alp,\sig)$. 
The next step introduces the letters $c_{\lam,b}$ into $W$ and $\sig(W)$.

\item  \textbf{Identifying a position in each block $dc_{b}^\lambda e$}. 
We represent $W \in M(B,\cX,\theta,\mu)$ by any word in $(B\cup \cX)^*$. 
For each letter $c_b$, we scan the word $\sig(W)$ from left to right and stop 
at each occurrence of a factor $dc_b^\lambda e$ where $\lam \in \Lam_b$ and $d \neq c_b \neq e$.
At the stop we do the following.
\begin{itemize}
\item If at least one of the $c_b$'s in this block is visible in $W$, then choose the left-most  corresponding visible position in $W$, and replace the label $c_b$ at this visible position by $c_{\lam,b}$. In $\sig(W)$, replace $dc_b^\lambda e$ by $dc_{\lam,b} c_b^{\lambda-1} e$.  
If no position of the $c_b$'s in this block is visible in $W$, then 
we make no change. 
 \end{itemize}
 Thus, from left to right, we transform the word $W$ into an element 
 $W' \in M(B,\cX',\theta,\mu)$ and simultaneously $\sig(W)$ into an element
 $\sig'(W')\in M(B).$ We obtain a new state
 $V'=  (W',B,\cX,\theta,\mu)$ and we can follow the 
 arc $V\arc{h}V'$ where $h$ is the $\cX$-\morph defined by a renaming 
 $h(c_{\lam,b}) = c_b$. Note that 
 $\Abs{V} > \Abs{V'}$ since for each $c_{\lam,b}$ a factor 
 $c_{\lam,b}c_b$ appears in $W'$, so there are more letters visible in $W'$ than in $W$, which decreases the second component in the weight of an \exe.  At $V'$ we obtain a new solution $(\alp,\sig')$; and as usual,
we rename  $V'= (\tau(W),B,\cX',\theta',\mu')$ as $V = (W,B,\cX,\theta,\mu)$ and rename the \solu as $(\alp,\sig)$.

Due to partial commutation we have the following: 
if a factor $f \in d\os{c_b,c_{\lam,b}}^\ell e$ occurs in $\sig(W)$ with $d,e \notin \os{c_b,c_{\lam,b}}$, then we have 
$\ell = \lam \in \Lam_b$, and $f= dc_{\lam,b} c_b^{\lambda-1} e \in M(B,\es,\theta,\mu)$. Moreover, if $\theta(X) = c_b$, then $X$ commutes with the letter $c_b$, but $X$ does not commute with any $c_{\lam,b}$.

\item \textbf{The block compression}.
As long as there exists a letter $c_b$ which occurs in $\sig(W)$, perform the following loop, which also finishes the block compression.  During the following loop we maintain the 
invariant: if $dc_{\lam,b} c_b^\ell e$ and $d'c_{\lam, b} c_b^{\ell'} e'$ are factors of 
$\sig(W)$ with  $d\neq c_b \neq e$ and $d'\neq c_b \neq e'$, then 
$\ell= \ell'$ and $\sig(W)$ contains a factor $\ov d\, \ov{c_{\lam,b}} \, \ov {c_b}^{\,\ell}\, \ov e$ as well. During the loop we perform various times a renaming in order to keep the notation $V$ and $(\alp,\sig)$ at the current states. 
Initially we define a list 
$$\Lam_B= \set{b\in B}{\Lam_b \neq \es}.$$

\medskip
\noindent {\bf while $\Lam_B \neq \es$ do}
\begin{enumerate}
\item For some $b\in \Lam_B$ remove $b$ and $\ov b$ from $\Lam_B$; 
\item Let $c = c_b$ and for all $\lam \in \Lam_b$ abbreviate $c_{\lam,b}$ as $c_{\lam}$.
\item {\bf while $\abs{\sig(W)}_{c} \geq 1$ do}
\begin{enumerate}
 \label{hugo} \item For all $X$ with $\theta(X) = c$ where $\abs{\sig(X)}$  is odd, follow a substitution transition of type $X \mapsto cX$. Hence, we may assume that $\abs{\sig(X)}$ is even for all $X$ with $\theta(X) = c$. 
\item Remove all $X$ {}from $\cX$ where $\sig(X)=1$. Observe, if 
there remains a variable $X$ with $\theta(X) =c$, then $\sig(W)$ contains a factor $c^2$. 
\item For all $c_\lam$ where $\sig(W)$ contains a factor 
$dc_\lam c^\ell e$ where $d\neq c \neq e$ and $\ell$ is odd, follow a compression \tra with $h(c_\lam) = c c_\lam$. 

In order to see that this is possible observe that for every occurrence of such a factor $dc_\lam c^\ell e$ 
there are only two possibilities. Either none of the positions of $c_\lam c^\ell$ are visible in $W$, or the position of $c_\lam$ is visible in $W$. 
Moreover, $c$ commutes with $c_\lam$ and with all $X$ where $\theta(X) =c$; and $\abs{\sig(X)}$ is even for those $X$. Thus, wherever $c_\lam$ is visible
in $W$, the factor $cc_\lam$ is visible in $W\in M(B,\cX,\theta,\mu)$. 

Still, we need to be more precise in order to guarantee a weight reduction.  
The $\cX$-\morph defined by $h(c_\lam) = c c_\lam$ leads to new element
$W'\in M(B,\cX,\theta,\mu)$ and a new \solu $(\alp h , \sig')$. In case 
that no letter $c$ occurs in $\sig'(W')$ anymore, the letter $c$ and the type becomes useless. Thus, if $\abs{\sig'(W')}_{c} = 0$, then we actually follow 
 a compression \tra
 $$V \arc h (W',B',\cX,\theta',\mu)$$
 where $B' = B \sm\os{c, \ov c}$ and hence $\abs{\theta'} < \abs {\theta}$. 
 Nevertheless $\Abs{V} > \Abs{V'}$ since $\abs{W'} < \abs{W}$ due to compression.
\item If there exists  a variable $X$ with $\theta(X) = c$, then we know 
$\sig(X) = c^2c^\ell$ where $\ell$ is even. We follow a \subst arc
defined by $X\mapsto c^2X$ 
in order to guarantee that a factor $c^2$ becomes visible in $W$.  
\item Due to the previous steps: either we have $c\notin B$ or 
$W$ contains a visible factor $c^2$. In the first case, we skip this step. 
Thus, we assume that $W$ contains a visible factor $c^2$.
Now, if $\sig(W)$ contains a factor
$dc_\lam c^\ell e$ where $d\neq c \neq e$, then $\ell$ is even; and
if $\theta(X) = c$, then $\sig(X) = c^{j}$ and $j$ is even, too.
Thus we can follow a compression \tra defined by $h(c) = c^2$. This leads to a new equation $W'$ with $h(W') = W$ and new solution $\sig'(W')$ and the 
number of occurrences of $c$ and $\ov c$ is halved. 
Note that $\Abs{V}> \Abs{V'}$ since $W$ contains a factor $c^2$. Hence, 
$\abs{W}> \abs{W'}$. 
Rename the parameters
to $V,W,B,\cX,\theta,\mu, \alp, \sig$.  
\end{enumerate}
\noindent{\bf endwhile}
\item Rename all $c_\lam$ by $c_{\lam,b}$.
\end{enumerate}
\noindent{\bf endwhile}
\end{enumerate}

\subsubsection*{Space requirements for the block compression}\label{subsec:spaceBlock}
Let us show that the block compression can be realized inside $\cA$. 
\begin{lemma}\label{block_comp_space}
 Let $V=(W,B,\cX,\es,\mu)$ be the state after preprocessing, when we enter ``block compression'', 
  and let $V'=(W',B',\cX',\es,\mu')$ be the state at the end of block compression. 
Then $V'$, as well as all intermediate states between $V$ and $V'$, are in $\cA$. Moreover, $\abs{W'} \leq \pcc$.
\end{lemma}
\begin{proof}
At the end of block compression we have $\cX'\sse \cX$, and each visible position of the new letter  $c_{\lam,b}$ occupies a position where some letter $b$ was visible in $W$. 
Thus, $\abs{W'}\leq \abs{W} \leq \pcc$. 

To show that the procedure stays inside $\cA$ we calculate the maximum length of an intermediate equation during the process.
We start block compression with $\abs{W}\leq \pcc$, and $\abs\cX \leq 4n$. 
In step (3) we add at most $8n$ new  variables $X'$ and at most $8n$ constants (we may substitute a variable $X$ by $aX'XX''b$ in the case that $\sig(X)=a^\ell wb^{\ell'}$).
So the length of the intermediate equation at this step is at most $\pcc+16n=120n+\abs\Winit$.
The only other step of block compression that adds length to the equation during the inner while-loop in step (5).

We start this loop with$\abs W\leq 120n+\abs\Winit$ and with at most $8n$ typed variables (the variables that were added in step (3)). We perform the loop at step (5c) with one letter $c\in \Lam_B$ fixed.

In step (i) we pop at most one $c$ letter  for each typed variable, and in step (ii) 
we pop $c^2$ for each typed variable, so we add  at most $3\cdot 8n=24n$ $c$'s, 
and then in step (v) we halve the number of $c$'s, so overall we add at most $12n$ 
$c's$. We repeat this loop until all $c$'s are eliminated.
In each iteration we add at most $24n$ new $c$ letters, but then divide the total 
number of $c$ letters by $2$. 
If we just consider the number of new $c$ letters added from the start of the while 
loop, we see that after each iteration the number of new $c$ letters remaining is at most:
\[ \begin{array}{|c|c|c|c|c|}
\hline
\text{iteration} & \text{number before } & \text{number added} & \text{number before} & \text{number after}\\
\text{} & \text{step (i)} &  & \text{step (v)} & \text{step (v)}\\
\hline
1 & 0 & 24n & 24n & 12n\\
2 & 12n & 24n & 36n & 18n\\
3 & 18n & 24n & 42n & 21n\\
4 & 21n & 24n & 45n & 23n\\
5 & 23n & 24n & 48n & 24n\\\hline
\end{array}\]
Thus the total length of $W$ is never more than 
\begin{equation}\label{eq:longestW}
120n+6\abs\Winit+48n=168n+6\abs\Winit
\end{equation} 

Since this call of the inner while-loop eliminates all occurrences of the letter $c$, at the end of each 
call the length of $W$ returns to being bounded above by $120n+6\abs\Winit$, when we 
repeat the while-loop at step (5c) for another constant in $\Lam_B$, until $\Lam_B=\es$.
Thus all states are in $\cA$. 
\end{proof}

For the final state $V'= (W',B',\cX',\es,\mu')$ 
the type relation is empty. If $V'$ is small, that is, $\abs{W'} \leq \bcc$, then \prref{lem:complett} is shown. 
Thus, without restriction we again have 
$$\bcc < \abs{W'} \leq \pcc.$$

\subsubsection*{Pair compression}\label{subsec:pairco}
After block compression we run {\em pair compression}, 
following essentially the  formulation of Je\.z's original procedure \cite{Jez16jacm}. 
We start a pair compression at a state $V_{p}=(W,B,\cX,\es,\mu)$ where we have:
\begin{itemize}
\item $\abs{W}_b \geq 1$ for all $b \in B\sm A$. 
\item $\bcc < \abs {W} \leq \pcc$. 
\item $W$ doesn't contain any proper factor $b^2$ with $b\in B\sm {\#}$.
\item The current \solu is denoted by $(\alp,\sig)$.
\end{itemize}
The goal of the process is to end at a state $V_{q}=(W'',B',\cX'',\es,\mu'')$
with $\abs {W''} \leq \bcc$ by some path satisfying the forward property and without 
increasing the weight. Moreover, there will be no types in this phase. Note that the constraints make sure that
$\sig(X)$ does not contain any factor $a\ov a$, but we cannot rule out that 
$W$ contains such factors. However, the number of  $a\ov a$ factors remains 
bounded by $\abs\Winit$, since they can only occur after preprocessing $\Winit$.

Consider all partitions $B\sm \os \# = L \cup R$ such that 
$b\in L \iff \ov b \in R$. Note that there is no overlap between  
factors $ab,cd \in LR$ unless $ab = cd$. Moreover
$$ab\in LR \iff \ov b \ov a \in LR.$$
For each choice of $(L,R)$ we count the number positions in $W$ where 
some factor $ab\in LR$ with $\ov a \neq b$ begins. We intend to compress all these factors into single letters.

\begin{remark}\label{rem:choiceLR}
We choose and fix one of the partitions $(L,R)$ such that the number of factors $ab\in LR$ 
in $\sig(W)$ such that $\ov a \neq b$ and at least one of $a$ or $b$ visible is maximal.
\end{remark}

We say that $ab\in LR$ is \emph{crossing} if $W$ contains either a factor 
$aX$ with $\sig(X) \in bB^*$ or a factor 
$\ov b X$ with $\sig(X) \in \ov a B^*$ (or both). 
In the first phase we run the following procedure. 

\noindent
\textbf{Uncrossing}.
Create a list $\cL= \set{X\in {\cX}}{\exists b \in R: \sig(X)\in bB^*}$.\\
For each $X\in \cL$: 
\begin{itemize}
\item choose $b \in R$ such that  $\sig(X)\in bB^*$ and  follow a \subst \tra $X \mapsto bX$. 
\end{itemize}
This concludes the ``uncrossing''; and, as done previously we rename the parameters to $V,W,B,\cX,\mu,\alp,\sig$. 

Above, when we follow $X \mapsto bX$ with $b \in R$, then automatically 
$\ov X$ is replaced with $\ov X \, \ov b$, and $\ov b \in L$. 
We also have $\os{X,\ov X} \sse \cL$ \IFF $\sig(X) \in bB^*a$ for some $ab\in LR$.
In that case we actually substituted $X$ by $bXa$ and $\ov X$ by 
$\ov a \ov X\,  \ov b$. Recall that we have at most $4n$ variables in $W$.
Thus, at this stage we have: 
\begin{equation}\label{eq:uncross}
\abs W \leq \pcc + 8n=112n+\abs\Winit
\end{equation}

The second phase begins with creating a list 
$\cP=\set{ab\in LR}{\ov a\neq b}$. After that we run the following while-loop. 
\medskip

\noindent\textbf{while $\cP \neq \es$ do}
\begin{enumerate}
\item Define $$B'= A\cup\set{a\in B}{\abs{W}_a \geq 1 \vee \exists X\in \cX: \sig(X) \in aB^*}.$$
If $B'\neq B$, then follow a \subst \tra $V\arc \eps (W,B',\cX,\es,\mu)$
where the label $\eps = \id{C^*}$ yields the inclusion of 
$M(B',\es,\es,\mu)$ into $M(B,\es,\es,\mu)$.
Rename the parameters to $V,W,B,\cX,\mu,\alp,\sig$.
\item Select and remove some pair $ab$ in $\cP$. 
If $ab$ does not occur as a factor in $W$, then do nothing, else perform the next steps. 
\item Choose a fresh letter $c = c_{ab} \in C\sm B$ with $\mu(c) = \mu(ab)$ and 
let $B''= B \cup \os{c, \ov c}$. Define an $\cX$-\morph 
$$h: M(B'',\cX,\es,\mu')\to M(B,\cX,\es,\mu)$$
 by $h(c) = ab$. 
\item Replace in $W$  all factors $ab$ by $c$ and all factors $\ov b \ov a$ by $\ov c$. Let $W'\in (B'\cup \cX)^*$ be the new word and 
$V'= (W',B'',\cX,\es,\mu')$ be the new state. We have $W = h(W')$; and hence 
there is a compression transition 
$$V \arc h V'.$$
\item Follow the compression transition 
$V \arc h V'$;  and rename the parameters to $V,W,B,\cX,\mu,\alp,\sig$. 
\end{enumerate}
\noindent\textbf{endwhile}\\ 

\begin{lemma}\label{lem:prwhlo}
During the while-loop for pair compression
the following properties hold. 
\begin{enumerate}
\item After the first step, where the new alphabet $B'$ is created
(and then renamed as $B$) we have $\abs B \leq \abs W + 2$. 
\item No factor $ab\in LR$ ever becomes crossing.
\item At each step where we move from state $V$ to $V'$ we have $\Abs V > \Abs{V'}$.
\item Each \tra satisfies the forward property. 
\end{enumerate}
\end{lemma}

\begin{proof}
\begin{enumerate}
\item In the first step inside the loop, when the new alphabet $B'$ is created, we have
$\abs{B'} \leq \abs W$. Therefore, after the first renaming,  we have
$\abs B \leq \abs W$. When we define $B''$, we add two new letters. Hence, 
we obtain $\abs{B''} \leq \abs W + 2$, which yields, after renaming, 
$\abs B \leq \abs W + 2$. This property persists during subsequent loops. 
\item We have to show that no factor $ab\in LR$ ever becomes crossing.
To see this, consider the alphabet reduction by 
following the \tra $V\arc \eps (W,B',\cX,\es,\mu)$ with $B'\neq B$.
It involves  replacing every letter $a\in B\sm B'$ by $\alp(a)$ according to \prref{lem:alpred}. The potential problem is that 
we might have $a\in L$, but  $\alp(a)$ starts with a letter in $R$, so we might create new $LR$ factors. However as $B'$ contains all letters $a$ 
where $\sig(X) \in aB^*$ for some $X$, we never introduce any new crossing pairs. 
\item The assertion $\Abs V > \Abs{V'}$ is trivial.
\item The \tra $V\arc \eps (W,B',\cX,\es,\mu)$ with $B'\neq B$ satisfies the forward property by \prref{lem:alpred}. 
In order to see that $V\arc h V'$ satisfies the forward property when 
we have $h(c) = ab$ we proceed as follows. As done for $W$, also replace in $\sig(W)$  all factors $ab$ by $c$ and all factors $\ov b \ov a$ by $\ov c$. Since $ab$ is not crossing, we find a 
$B'$-\morph $$\sig':M(B',\cX,\es,\mu')^*\to M(B',\es,\es,\mu')$$ such that 
$\sig(X) = h\sig'(X)$ for all variables $X$. 
Thus, we obtain $(\alp h, \sig')$ as a \solu at $V'$
\end{enumerate}
\end{proof}

\begin{lemma}\label{pair_comp_small}
Let  $V_{p}=(W, B,{\cX}, \es, \mu)$ be a state in $\cA$ with a \solu $(\alp,\sig)$ where $ \bcc < \abs W \leq \pcc$ such that $W$ doesn't contain any factor $d^2$ for $\# \neq d\in B$. Let $(L,R)$ be the partition with $B\sm \os \# = L \cup R$ according to the choice made in
 \prref{rem:choiceLR}. Then pair compression on $V_p$  leads to a state
 $V_{q}=(W'', B',{\cX}, \es, \mu'')$  with  $|W''|\leq \bcc$, that is, the state $V_q$ is small. Moreover, the intermediate steps of the pair compression algorithm are performed within $\cA$.
\end{lemma}
\begin{proof}
Recall that the NFA $\cA$ is trim. Hence, there is a path 
$$V_0 \arc {h_1} \cdots V_{p-1} \arc  {h_p} V_p$$
from an initial state with the appropriate $\mu$ to $V_p$. Let $V_{i}=(W_i, B_i,{\cX_i}, \theta_i, \mu_i)$. We perform the following {\em marking} process. The idea is that we wish to mark all constants in the $W_i$ which could possibly give rise to a factor $a\ov a$ in $W$. These factors can arise in exactly two ways: the initial equation may be unreduced to start with, or from a substitution (for example, we may have $aX$ or $YZ$  factors of the initial equation and we pop $X\ra aX$ or  $Y\ra Ya, Z\ra \ov  a Z$).
\begin{enumerate}\item  In  $W_0 = \Winit$ we mark all letters (both constants and variables).
\item If $V_{i-1} \arc  \eps V_i$ is a \subst \tra,  $W_i = \tau(W_{i-1})$ and the positions with constants in 
$W_{i-1}$ are mapped to positions with constants in $W_i$. 
We mark constants in $W_i$ that come from marked constants in $W_{i-1}$, and if $\tau(X)\in a\Gamma^*$ and $X$ is marked in $W_{i-1}$, we mark the newly added $a$ on the left of the variable $X$ in $W_i$, and leave  $X$ unmarked. If $\tau(Y)=Y$ and $Y$ is marked in $W_{i-1}$, we leave $Y$ marked in $W_i$. Note that in this way each marked variable gives rise to exactly one marked letter.
\item If $V_{i-1} \arc  h V_i$ is a compression \tra, then we have $h(W_i) = W_{i-1}$. 
Mark a constant $c$ in  $W_i$ if it is mapped by $h$ to an occurrence of a factor containing a marked position in $W_{i-1}.$
\end{enumerate}
Note that since the pair compression procedure  is always preceded by the preprocessing step above, we can assume that  every variable $X$ in $\Winit$ has been replaced by $aX$ where $a$ is marked, so in $V_p$ the word $W$ contains at most $\abs\Winit$ marked constants and no marked  variables.

When we run the pair compression procedure on $W$ we cannot compress pairs $a\ov a$, or pairs containing variables. If we now mark all variables present in $W$, then we are allowed to compress any pairs of letters in $W$ that are unmarked. After marking the variables we have at most $2\abs\Winit$ marked letters in $W$.

Let us factor the word $W \in (B \cup \cX)^*$ as 
$W = x_{0}u_{1}x_{1}\cdots u_{\ell}x_{\ell}$, where $\ell$ is chosen to be maximal that for all $1 \leq i \leq \ell$ we have:
\begin{enumerate}
\item $x_i\in (B\cup \cX)^* $.
\item $u_{i}\in (B\sm \os \#)^*$ and $u_i$ doesn't contain any marked position. 
\item The length of each $u_{i}$ is exactly $3$.
\end{enumerate}
The factorization enjoys the  following properties. 
\begin{itemize}
\item Since all $\#$'s are marked, we have $x_0\neq 1 \neq x_\ell$.
Some other $x_i$ can be empty.
\item Since we require $\abs{u_{i}} = 3$ it may be that $x_i$ contains 
for each marked position also two unmarked position. The exception is the first position in $x_0$. Hence, we obtain
$$\sum_{0 \leq i \leq \ell}\abs{x_i} \leq 3(2\abs{\Winit})-2\leq 6\abs\Winit.$$
\item Since $\abs W - 6\abs{\Winit}>96 n$, the previous line yields 
$$\ell > 32n.$$
\end{itemize}

Consider the word $W'$ which was obtained via the \subst transitions, but before the compression of factors $ab\in LR$ into single letters. 
The increase in length, which is $\abs{W'} - \abs{W}$, comes from the
substitution transitions  $X\mapsto bX, \ov X\mapsto \ov X \, \ov b$ with $X \in \cL$, so the length goes up by at most $8n$.
Note that the $u_i$ factors do not change, only the $x_i$ factors do.
Hence $W'$ has the factorization 
$W' = y_{0}u_{1}y_{1}\cdots u_{\ell}y_{\ell}$ with  $y_i \in (B \cup \cX)^*$ and  
\begin{align}\label{eq:loWw}
\abs{y_{0}\cdots y_{\ell}}\leq \abs{x_{0}\cdots x_{\ell}} +8n.
\end{align}
Finally, let $W''$ be the word obtained after pair compression has been performed. The word $W''$ is the compression of some
 word 
$y_{0}v_{1}y_{1}\cdots v_{m}y_{m}$
where each $v_{i}$ is the result of the compression restricted to $u_{i}$.

Each $u_{i}$ can be written as $u_{i}= abc$ with $a,b,c \in B$. 
Since $W$ did not contain any proper factor $d^2$ with 
$d \in B$ by hypotheses (and as we have performed block compression first), 
we know $a \neq b \neq c$. Moreover, we cannot have $\ov a = b$ or 
$\ov c = b$ because in every occurrence of  $b\ov b$ in $W$ at least one position is marked. 

Assume for a moment that membership to $L$ or $R$ was defined uniformly at random. That is for each $\# \neq a \in B$ the probability for 
$a\ov a \in LR$ is $\frac12$ and independent of the other events 
``$b\ov b \in LR$'' for $a \neq b$.

There are two possibilities: either $b\in L$ or 
$b\in R$. 
In the first case, either $c\in R$ or $c\in L$, and in the second case either $a\in L$ or $a\in R$. Each event 
$bc\in LR, bc\in LL, ab\in LR, ab\in RR$
has probability $\frac14$, so with probability  $\frac12$ one pair in the factor $u_i$  is compressed: thus the expected length of a factor $v_i$ is  $\E{\abs{v_{i}}} = \frac32 + \frac22 =\frac52$.
By linearity of expectation, we obtain
\begin{align}\label{eq:looW}
\E{\abs{v_{1}\cdots v_{\ell}}} = \tfrac{5}{2}\ell.
\end{align}
Thus if the partition $(L,R)$ were chosen at random, we expect the length of the word $u_{1}\cdots u_{\ell}$ to decrease from $3\ell$ to $\tfrac{5}{2}\ell$ or less, that is, we expect at least $\tfrac{1}{6}\ell$ factors $u_i$ are compressed   (each $v_i$ has length either $2$ or $3$). 
But in \prref{rem:choiceLR} we made the best choice of  compressing a maximal number of pairs in $W'$. This means at least $\tfrac{1}{6}\ell$ factors of $W'$ are compressed. Hence, for the actual pair compression, we may estimate the length of $W''$  as follows.

\begin{alignat*}{2}
\abs{W''} 
&\leq \abs{x_{0}\cdots x_{\ell}} +8n  + \tfrac52\ell
&\qquad\text{ since $\tfrac{1}{6}\ell$ factors are compressed} 
\\ &= \abs {W} +8n - \tfrac{\ell}2 &\qquad\text{since 
$\abs{W} = \abs{x_{0}\cdots x_{\ell}} +3\ell$}
\\ & \leq \abs {W} -8n &\qquad\text{ since $\ell > 32n$ }
\\ & \leq \bcc &\qquad\text{ since $|W|\leq \pcc$.}
\end{alignat*}

Since $\abs{W''} \leq \bcc$, the last state $V_{q}= (W'',B',\cX,\es,\mu'')$ 
is small. 
\end{proof}

A linear bound on the size of $C$ is evident from the proofs above and an explicit bound is given next. 
Thus, we have shown \prref{lem:complett}.

\subsubsection{The size of the extended alphabet $C$: the choice of $\kappa$}\label{subsec:sizeC}

The longest equation $W$ we needed to establish completeness 
occurs  during block compression, where we found that $\abs W\leq 168n+6\abs\Winit$ (\ref{eq:longestW}). Combining this with  $\abs\Winit\leq 6n$   (\ref{eq:UVwinit})
we obtain
\begin{equation}\label{eq:CsizeOne}
\abs W \leq 
168n+36n=204n.
\end{equation} 
The largest
alphabet we ever needed during  block and pair compression was less than 
$$3 \cdot (\abs \Apos + \abs W) \leq 3\cdot (n+204n)= 3\cdot 205n=615n.$$  Thus, we can choose $\kappa$ such that 
\begin{equation}\label{eq:cblowup}
\abs {C} = \kappa \cdot n =  615n. 
\end{equation}

\subsubsection{Finishing the proof of \prref{thm:alice} in the monoid case}\label{subsec:fpoalice}
\prref{lem:complett} implies \prref{prop:completeness} by the reduction in \prref{subsec:popco}. 
This in turn proves (\ref{eq:alice}) in \prref{thm:alice} in the monoid case $\MMA = A^*$.  Clearly, $\{(h(c_1) \lds  h(c_m))\in C^*\times \cdots \times C^*\mid h\in L(\cA)\}$ is empty \IFF $L(\cA)= \es$. It remains to show that
$\cA$ contains a directed cycle \IFF $(U,V)$ has infinitely many solutions. 
If there is no cycle, then $L(\cA)$ is finite and $(U,V)$ can have only finitely many solutions. The converse has been shown in \prref{cor:infsol}. 

\section{Proof of Theorem~\ref{thm:alice} in the group case: $\M(A) = \FGA$}\label{sec:alicegroupie}

The proof is a reduction to the monoid case.
Recall that $A=\Apone \cup \{\#\}$, $\F$ is the subset of reduced words in $\Apone^*$, and $\pi: A^*\to \FGA$ is the canonical projection. 

We start with an equation $(U,V)$ in the free group $\FGA$, where $U,V\in (A\cup \cX)^*$, 
$\cX= \os{X_1,\ov{X_1}, \ldots, X_m,\ov{X_m}}$, and solutions are $A$-\morph{s} 
$\sig: (A\cup \cX)^* \to \F$ such that 
$\pi\sig(U)= \pi\sig(V)$. 
 In a first phase we transform the equation $(U,V)$ into a system of triangular equations, where \emph{triangular} means $1 \leq \abs{UV}\leq 3$. We may assume $UV\neq 1$. If $\abs{UV} \leq 3$, then the equation is already triangular. Hence, let us assume $\abs{UV}\geq 4$.
Since we are in the group case we may also assume $\abs V = 1$. 
Write $U=x_1\cdots x_p$ with $x_i \in A\cup \cX$ and $p\geq 3$. Next, we introduce a new variable $X$ and replace 
$x_1\cdots x_p=V$ by the system 
$$x_1\cdots x_{p-1}=X \wedge Xx_p = V.$$
We iterate until the system is triangular. The procedure introduces more variables, but it does not change the set of solutions. 
More formally, if $\set{(U_i,V_i)}{1 \leq i \leq t}$ is the system of triangular equations we obtained above, then 
\begin{align*}
 \{(\sig(X_1) \lds  \sig(X_m))&
 \in \F\times \cdots \times \F\mid \pi\sig(U)=\pi\sig(V)\} \\ 
 =  \{(\sig(X_1) \lds  \sig(X_m))&
 \in \F\times \cdots \times \F\mid \forall 1 \leq i \leq t:\; \pi\sig(U_i)=\pi\sig(V_i).\}
\end{align*}
The crucial step in our reduction is to switch from \solu{s} over free groups to 
\solu{s} over free monoids with \invol. We do this using the following lemma, whose geometric interpretation is simply that the Cayley graph of a free group (over standard generators) is a tree.

\begin{lemma}\label{lem:caytree}
Let $x,y,z$ be reduced words in $\Apone^*$. Then $xy=z$ holds in the group $\FGA$ ({\em i.e.} $\pi(xy)= \pi(z)$)
\IFF there are reduced words $P,Q,R$ in $\Apone^*$ such that
$x= PR$, $y = \ov RQ$, and $z =  PQ$ holds in the free monoid $\Apone^*$. 
 \end{lemma}

\begin{figure}[h]
\begin{center}
\begin{tikzpicture}[outer sep=0pt, inner sep = 1pt, node distance = 8pt]
\node[circle, fill]  (a) at (0,2) {};
 \node[circle, fill] (c) at (0,0) {};   
 \node[circle, fill] (x) at (-2,-1) {};
 \node[circle, fill] (z) at (2,-1) {};

 \begin{scope}[very thick,decoration={
    markings,
    mark=at position 1 with {\arrow{>}}}] 
 \draw[thick, postaction={decorate}] (a) -- node[right =.1] {$P$} (c); 
  \draw[thick, postaction={decorate}] (c) -- node[above =.1] {$R$} (x);
   \draw[thick, postaction={decorate}] (c) -- node[above =.1] {$Q$} (z);  

 \end{scope}
 
 \node[left of  = x, node distance = 10pt] {$x$};
 \node[right of  = z, node distance = 10pt] {$z$};
 \node[left of  = a] {$1$};
 \path (-1.3,-1) edge [bend left, ->] node [below] {$y$}   (1.3,-1);
\end{tikzpicture}
\caption{Paths corresponding to geodesic words for $x,y,z$ with $xy=z$ in the Cayley graph of $\FG \Apos$ with standard generators, as in \prref{lem:caytree}. The geodesics to vertices $x$ and $z$ split after an initial path labeled by  $P$.
}
\label{fig:xyz}
\end{center}
\end{figure}
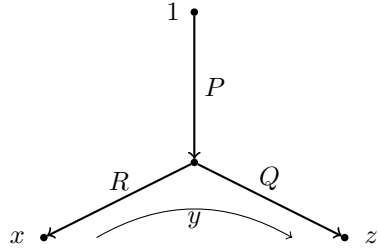

\begin{proof}
 The direction from right to left is trivial, whether or not 
 $P,Q,R$ are reduced. For the other direction there are two cases.
 First, $xy$ is a reduced word. Then we can choose
 $P=x$, $R=1$, $Q=y$, and we are done. Second, we have $x=x'a$ and $y =\ov a y'$ for some letter $a \in \Apone$, so $x'y'=z'$ holds in the group $\FGA$.
 By induction, there are  reduced words $P,Q,R'$ with $x'=PR', y'=\ov {R'}Q, z=PQ$ in $\Apone^*$. We can define
 $R= R'a$, which is reduced due to the equation $x= x'a = PR'a$ and the fact that $x$ is reduced. The result is now immediate. 
  \end{proof}

The  consequence of \prref{lem:caytree} is that with the help of fresh variables $P,Q,R$  we can substitute every equation  
$xy= z$ with $x,y,z\in \os 1 \cup \Apone \cup \OO$ in $\FGA$ 
by the following three word equations to be solved over a free monoid with involution:
\begin{align}\label{eq:fg2fm}
x = PR,\qquad  y = \ov R Q,\qquad z = PQ.
\end{align}
More precisely, in the third phase of the transformation we replace each 
$U_i= V_i$, where $U_i= x_iy_i$ and $V_i=z_i$, by the three equations
\begin{align}
x_i = P_iR_i,\qquad  y_i = \ov{R_i} Q_i,\qquad z_i = P_iQ_i.
\end{align}
Thus, for $s=3t \leq 3\abs{UV}$ we obtain a new system of triangular word equations $\set{(U'_i,V'_i)}{1 \leq i \leq s}$
such that
\begin{align}\label{eq:alices} 
 \{(\sig(X_1) \lds  \sig(X_m))&
 \in \F\times \cdots \times \F\mid \pi\sig(U)=\pi\sig(V)\} \\ 
 \label{eq:alicess}
 =  \{(\sig(X_1) \lds  \sig(X_m))&
 \in \F\times \cdots \times \F\mid \forall 1 \leq i \leq s:\; \sig(U'_i)=\sig(V'_i).\}
\end{align}
Note that the \morph $\pi$ is not present in (\ref{eq:alicess}), since (\ref{eq:alicess}) refers to a system of equations over a free monoid with involution.

The final step is to encode the system $\set{(U'_i,V'_i)}{1 \leq i \leq s}$ into a single word equation $(U'',V'')$ over the free monoid $A^*$, by defining 
\begin{align*}
U'' = U'_1\# \cdots \# U'_s\\
V'' = V'_1\# \cdots \# V'_s.
\end{align*}
Thus we have deterministically reduced the equation
$(U,V)$ to the equation $(U'',V'')$, where
\begin{align*}\label{eq:alicessses} 
\abs{U''V''}  \leq 15\abs{UV}
\end{align*}
since each $U_i'V_i'$ has length at most $3$ and we have inserted  $2s-2$ copies of the letter $\#$.
This finishes the proof of \prref{thm:alice} for the group case. 
\begin{remark}\label{rem:crazy_bound_for_free_group}
Since the length of the word equation obtained from a free group equation of length $n$  is at most $15n$, an upper bound for the size of the alphabet  $C$ in the statement of \prref{thm:alice} in the free group case is
 $615\cdot 15n  = 9225n$.\end{remark} 
  
 \section{Example of preprocessing, block and pair compression procedures}\label{sec:blockexample}

 We conclude with a demonstration of the 
 procedures described in \prref{subsec:compness} with a simple example.
 Suppose we have a single equation $(U,V)$ in a free monoid with involution 
 with \[U=XaYbaXP \;\text{ and }\;  V=bYb^3ZQ.\]
 For simplicity we have chosen an equation with no involuted letters.
 Suppose also that we know a solution \[\sig(X)=b^5, \sig(Y)=b^4a, \sig(Z)=bab, \sig(P)=ab^3a, \sig(Q)=ab^5ab^3a.\]
 We depict the situation as follows:
  \[ \overunderbraces{&\br{5}{X} & &\br{5}{Y} &&& \br{5}{X} &\br{5}{P}}
{  &b  \;&\; b  \;&\; b \;&\; b \;&\; b \;&\; a   \;&\; b \;&\; b \;&\; b \;&\; b \;&\; a \;&\; b \;&\; a \;&\; b \;&\; b \;&\; b \;&\; b \;&\;  b\;&\; a \;&\; b \;&\; b \;&\; b   \;&\; a
} 
{&&\br{5}{Y}&&&&\br{3}{Z}&\br{11}{Q}}.   
\]

 For simplicity, we will ignore the rest of the word $\Winit$, 
  and focus just on the factor $U\#V$. 
 
We first follow the preprocessing step on page \pageref{page:preprocessing}. In this case we pop the first and last letter of each variable, to obtain:
  \[ \overunderbraces{&&\br{3}{X} &&&&\br{3}{Y} &&&&& \br{3}{X}&&&\br{3}{P}}
{  &b  \;&\; b  \;&\; b \;&\; b \;&\; b \;&\; a   \;&\; b \;&\; b \;&\; b \;&\; b \;&\; a \;&\; b \;&\; a \;&\; b \;&\; b \;&\; b \;&\; b \;&\;  b\;&\; a \;&\; b \;&\; b \;&\; b   \;&\; a
} 
{&&&\br{3}{Y}&&&&&&\br{1}{Z}&&&\br{9}{Q}}.   
\]

 Next we enter block compression. In step (1) we compute $\Lambda_a=\es, \Lambda_b=\{4,5\}$. Note that $3\not\in \Lambda_b$ since the factor $b^3$ is completely inside $P$ and $Q$ so is not visible. The block compression process will not touch this factor.
 We also compute $\cX_a=\es$ and $\cX_b=\{X,Y\}$. Note that $P\not\in \cX_b$  since it is preceded by $a$ in $W$.

 Step (2) introduces the fresh letters $c_b,c_{4,b}, c_{5,b}$, and renames the letters $b$ that are part of a visible block of length at least 2 as $c_b$:
  \[ \overunderbraces{&&\br{3}{X} &&&&\br{3}{Y} &&&&& \br{3}{X}&&&\br{1}{P}}
{  &c_b  \;&\; c_b  \;&\; c_b \;&\; c_b \;&\; c_b \;&\; a   \;&\; c_b \;&\; c_b \;&\; c_b \;&\; c_b \;&\; a \;&\; b \;&\; a \;&\; c_b \;&\; c_b \;&\; c_b \;&\; c_b \;&\;  c_b\;&\; a \;&\;  b^3   \;&\; a
} 
{&&&\br{3}{Y}&&&&&&\br{1}{Z}&&&\br{7}{Q}}.   
\]

 In step (3) we split the variables $X\ra X'X,Y\ra Y'Y$,  then remove $X,Y$ since $\sig(X)=1=\sig(Y)$:
  \[ \overunderbraces{&&\br{3}{X'} &&&&\br{3}{Y'} &&&&& \br{3}{X'}&&&\br{1}{P}}
{  &c_b  \;&\; c_b  \;&\; c_b \;&\; c_b \;&\; c_b \;&\; a   \;&\; c_b \;&\; c_b \;&\; c_b \;&\; c_b \;&\; a \;&\; b \;&\; a \;&\; c_b \;&\; c_b \;&\; c_b \;&\; c_b \;&\;  c_b\;&\; a \;&\;  b^3   \;&\; a
} 
{&&&\br{3}{Y'}&&&&&&\br{1}{Z}&&&\br{7}{Q}}.   
\]
 Note that $Q$ does not belong to $\cX_b$, so it does not split even though $\sig(Q)$ starts with $c_b$.

 Step (4) renames one of the $c_b$ in each block in both $W$ and $\sig(W)$:
  \[ \overunderbraces{&&\br{3}{X'} &&&&\br{3}{Y'} &&&&& \br{3}{X'}&&&\br{1}{P}}
{  &c_{5,b}  \;&\; c_b  \;&\; c_b \;&\; c_b \;&\; c_b \;&\; a   \;&\; c_{4,b} \;&\; c_b \;&\; c_b \;&\; c_b \;&\; a \;&\; b \;&\; a \;&\; c_{5,b} \;&\; c_b \;&\; c_b \;&\; c_b \;&\;  c_b\;&\; a \;&\;  b^3   \;&\; a
} 
{&&&\br{3}{Y'}&&&&&&\br{1}{Z}&&&\br{7}{Q}}.   
\]

 We now enter the loop in step (5). We  write $c=c_b, c_\lambda=c_{\lambda, b}$:
  \[ \overunderbraces{&&\br{3}{X'} &&&&\br{3}{Y'} &&&&& \br{3}{X'}&&&\br{1}{P}}
{  &c_5  \;&\; c  \;&\; c \;&\; c \;&\; c \;&\; a   \;&\; c_4 \;&\; c \;&\; c \;&\; c \;&\; a \;&\; b \;&\; a \;&\; c_5 \;&\; c \;&\; c \;&\; c \;&\;  c\;&\; a \;&\;  b^3   \;&\; a
} 
{&&&\br{3}{Y'}&&&&&&\br{1}{Z}&&&\br{7}{Q}}.   
\]

 Since $\theta(X')=\theta(Y')=c$ we pop each to make the number of $c$ letters in each $\sig(X)$ even:
  \[ \overunderbraces{&&\br{2}{X'} &&&&&&\br{2}{Y'} &&&&& \br{2}{X'}&&&&\br{1}{P}}
{  &c_5  \;&\; c  \;&\; c \;&\; c \;&\; c \;&\; a   \;&\; c_4 \;&\; c \;&\; c \;&\; c \;&\; a \;&\; b \;&\; a \;&\; c_5 \;&\; c \;&\; c \;&\; c \;&\;  c\;&\; a \;&\;  b^3   \;&\; a
} 
{&&&&\br{2}{Y'}&&&&&&\br{1}{Z}&&&\br{7}{Q}}.   
\]
 Note that we have used the fact that $X',Y'$ commute with $c$ in the partially commutative monoid.

 We are now at part (d) of step (6). Since $c_4c^3$ is a factor where the number of $c$ letters is odd, 
 we follow the compression \tra $h(c_4)=c_4c$  to obtain:
    \[ \overunderbraces{&&\br{2}{X'} &&&&&\br{2}{Y'} &&&&& \br{2}{X'}&&&&\br{1}{P}}
{  &c_5  \;&\; c  \;&\; c \;&\; c \;&\; c \;&\; a   \;&\; c_4 \;&\; c \;&\; c \;&\; a \;&\; b \;&\; a \;&\; c_5 \;&\; c \;&\; c \;&\; c \;&\;  c\;&\; a \;&\;  b^3   \;&\; a
} 
{&&&&\br{2}{Y'}&&&&&\br{1}{Z}&&&\br{7}{Q}}.   
\]

 We now have all blocks of $c$ inside variables and in $W$ of even length, so we can finally follow the block compression \tra $h(c)=cc$ to reduce the number of $c$ letters by half:
    \[ \overunderbraces{&&\br{1}{X'} &&&&\br{1}{Y'} &&&&& \br{1}{X'}&&&\br{1}{P}}
{  &c_5  \;&\; c \;&\; c \;&\; a   \;&\; c_4 \;&\; c \;&\; a \;&\; b \;&\; a \;&\; c_5  \;&\; c \;&\;  c\;&\; a \;&\;  b^3   \;&\; a
} 
{&&&\br{1}{Y'}&&&&\br{1}{Z}&&&\br{5}{Q}}.   
\]

Since there are still $c$ letters remaining in $\sig(W)$ we repeat the loop, and after two more iterations of the loop we obtain:
    \[ \overunderbraces{& &&&&&&&&\br{1}{P}}
{  &c_5    \;&\; a   \;&\; c_4 \;&\; a \;&\; b \;&\; a \;&\; c_5 \;&\; a \;&\;  b^3   \;&\; a
} 
{&&&&\br{1}{Z}&&&\br{3}{Q}}.   
\]

 At this point we have removed all letters $c_b$ so the loop terminates. 
 We reduce the alphabet by removing $c_b$, and remove the types.
 Note that we keep each $c_{\lambda, b}$ since each letter represents a different length block of $b$'s, and therefore they are all different. Let us rename $c_{5,b}=d$ and $c_{4,b}=e$. So the equation is now:
     \[ \overunderbraces{& &&&&&&&&\br{3}{P}}
{  &d   \;&\; a   \;&\; e \;&\; a \;&\; b \;&\; a \;&\; d \;&\; a \;&\;  b \;&\;  b \;&\;  b   \;&\; a
} 
{&&&&\br{1}{Z}&&&\br{5}{Q}}.   
\]
 As promised, $W$ contains no proper factors $b^2$ for any $b\in B$, so we can start pair compression.

Suppose we choose a partition of $B\sm\{\#\}$ as $B_+=\{\ov a, b,d,e\}$ and $B_-=\{a, \ov b,\ov d,\ov e\}$ (we suppose  this choice is maximal according to  \prref{rem:choiceLR}).
 In step (1) of pair compression we introduce  fresh letters $c_{ba}, c_{da}, c_{ea}$, then in step (2) we create the list $\cL=\{Z, \ov P, \ov Q\}$. (We will continue to ignore involutions, and focus just on a factor of $W$ containing no involuted letters or variables).
 We perform uncrossing by popping $a$ from $Z$ and removing $Z$, and  since we follow $\ov P\ra \ov b\ov P$ then we also follow $P\ra Pb$, and similarly for $Q$, leading to:
     \[ \overunderbraces{& &&&&&&&&\br{2}{P}}
{  &d   \;&\; a   \;&\; e \;&\; a \;&\; b \;&\; a \;&\; d \;&\; a \;&\;  b \;&\;  b \;&\;  b   \;&\; a
} 
{&&&&&&&\br{4}{Q}}.   
\]

In step (3) we follow  compression transitions $h(c_{ba})=ba, h(c_{da})=da, h(c_{ea})=ea$ to obtain:
     \[ \overunderbraces{&&&&&\br{2}{P}}
{  &c_{da}   \;&\; c_{ea} \;&\; c_{ba} \;&\; c_{da} \;&\;  b \;&\;  b \;&\;  c_{ba}
} 
{&&&&\br{3}{Q}}.   
\]

This completes one round of the process. 
We then return to the preprocessing step, which gives:
      \[ \overunderbraces{&&&&&}
{  &c_{da}   \;&\; c_{ea} \;&\; c_{ba} \;&\; c_{da} \;&\;  b \;&\;  b \;&\;  c_{ba}
} 
{&&&&&\br{1}{Q}},   
\]
and then block compression would produce:
      \[ \overunderbraces{&&&&&}
{  &c_{da}   \;&\; c_{ea} \;&\; c_{ba} \;&\; c_{da} \;&\;  c_{2,b} \;&\;  c_{ba}.
} 
{&&&&&}   
\]

\section*{Acknowledgments}

We wish to thank Monserrat Casals-Ruiz, Artur Je\.z,   Ilya Kazachkov,  Markus Lohrey,  Alexei Miasnikov, Nicholas Touikan for helpful discussions.
We are particularly indebted to the referee for numerous suggestions which greatly improved the presentation.

\bibliographystyle{abbrv}
\bibliography{traces}

\newcommand{\Ju}{Ju}\newcommand{\Ph}{Ph}\newcommand{\Th}{Th}\newcommand{\Ch}{Ch}\newcommand{\Yu}{Yu}\newcommand{\Zh}{Zh}\newcommand{\St}{St}\newcommand{\curlybraces}[1]{\{#1\}}
\begin{thebibliography}{10}

\bibitem{Aho68}
A.~V. Aho.
\newblock Indexed grammars---an extension of context-free grammars.
\newblock {\em J. ACM}, 15:647--671, 1968.

\bibitem{Asveld1977}
P.~R. Asveld.
\newblock Controlled iteration grammars and full hyper-{AFL}'s.
\newblock {\em Information and Control}, 34(3):248 -- 269, 1977.

\bibitem{CiobanuDEicalp2015}
L.~Ciobanu, V.~Diekert, and M.~Elder.
\newblock Solution sets for equations over free groups are {EDT0L} languages.
\newblock In M.~Halld{\'o}rsson, K.~Iwama, N.~Kobayashi, and B.~Speckmann,
  editors, {\em Proc. 42nd International Colloquium Automata, Languages and
  Programming (ICALP 2015), Part~{II}, Kyoto, Japan, July 6-10, 2015}, volume
  9135 of {\em Lecture Notes in Computer Science}, pages 134--145. Springer,
  2015.
\newblock Journal version to appear 2016 in IJAC.

\bibitem{CiobanuDEarxiv2015}
L.~Ciobanu, V.~Diekert, and M.~Elder.
\newblock Solution sets for equations over free groups are {EDT0L} languages.
\newblock {\em ArXiv e-prints}, abs/1502.03426, 2015.

\bibitem{cl85}
M.~Clerbout and M.~Latteux.
\newblock Partial commutations and faithful rational transductions.
\newblock {\em Theoretical Computer Science}, 34:241--254, 1984.

\bibitem{DiekertJP2014csr}
V.~Diekert, A.~Je\.z, and W.~Plandowski.
\newblock Finding all solutions of equations in free groups and monoids with
  involution.
\newblock In E.~A. Hirsch, S.~O. Kuznetsov, J.~Pin, and N.~K. Vereshchagin,
  editors, {\em Computer Science Symposium in Russia 2014, {CSR} 2014, Moscow,
  Russia, June 7-11, 2014. Proceedings}, volume 8476 of {\em Lecture Notes in
  Computer Science}, pages 1--15. Springer, 2014.

\bibitem{dr95}
V.~Diekert and G.~Rozenberg, editors.
\newblock {\em The Book of Traces}.
\newblock World Scientific, Singapore, 1995.

\bibitem{EhrRoz77}
A.~Ehrenfeucht and G.~Rozenberg.
\newblock On some context free languages that are not deterministic {ET0L}
  languages.
\newblock {\em RAIRO Theor. Inform. Appl.}, 11:273--291, 1977.

\bibitem{eil74}
S.~Eilenberg.
\newblock {\em Automata, Languages, and Machines}, volume~A.
\newblock Academic Press, New York and London, 1974.

\bibitem{FerteMarinSenizerguesTocs14}
J.~Fert{\'e}, N.~Marin, and G.~S{\'e}nizergues.
\newblock Word-mappings of level $2$.
\newblock {\em Theory Comput. Syst.}, 54:111--148, 2014.

\bibitem{GinsburgRoz75}
S.~Ginsburg and G.~Rozenberg.
\newblock {T0L} schemes and control sets.
\newblock {\em Information and Control}, 27:109--125, 1975.

\bibitem{Jez16jacm}
A.~Je\.z.
\newblock Recompression: a simple and powerful technique for word equations.
\newblock {\em J. ACM}, 63(1):4:1--4:51, 2016.
\newblock Conference version in STACS 2013.

\bibitem{kel73}
R.~M. Keller.
\newblock Parallel program schemata and maximal parallelism~{I}. {F}undamental
  results.
\newblock {\em J. ACM}, 20(3):514--537, 1973.

\bibitem{KMIV06}
O.~{Kharlampovich} and A.~{Myasnikov}.
\newblock Elementary theory of free non-abelian groups.
\newblock {\em J. of Algebra}, 302:451--552, 2006.

\bibitem{maz77}
A.~Mazurkiewicz.
\newblock Concurrent program schemes and their interpretations.
\newblock {DAIMI Rep. PB}~78, Aarhus University, Aarhus, 1977.

\bibitem{mes97}
J.~Messner.
\newblock Pattern matching in trace monoids.
\newblock In R.~Reischuk, editor, {\em Proc. 14th Annual Symposium on
  Theoretical Aspects of Computer Science (STACS'97), L{\"u}beck (Germany),
  1997}, volume 1200 of {\em Lecture Notes in Computer Science}, pages
  571--582, Heidelberg, 1997. Springer-Verlag.

\bibitem{pap94}
{\Ch}.~H. Papadimitriou.
\newblock {\em Computational Complexity}.
\newblock Addison Wesley, 1994.

\bibitem{pin86}
J.-{\'E}. Pin.
\newblock {\em {Varieties of Formal Languages}}.
\newblock North Oxford Academic, London, 1986.

\bibitem{Plandowski06stoc}
W.~Plandowski.
\newblock An efficient algorithm for solving word equations.
\newblock In J.~M. Kleinberg, editor, {\em STOC}, pages 467--476. ACM, 2006.

\bibitem{raz87}
A.~A. Razborov.
\newblock {\em On Systems of Equations in Free Groups}.
\newblock PhD thesis, Steklov Institute of Mathematics, 1987.
\newblock In Russian.

\bibitem{raz93}
A.~A. Razborov.
\newblock On systems of equations in free groups.
\newblock In {\em Combinatorial and Geometric Group Theory}, pages 269--283.
  Cambridge University Press, 1994.

\bibitem{RozS86}
G.~Rozenberg and A.~Salomaa.
\newblock {\em The Book of {L}}.
\newblock Springer, 1986.

\bibitem{rs97vol1}
G.~Rozenberg and A.~Salomaa, editors.
\newblock {\em Handbook of Formal Languages}, volume~1.
\newblock Springer, 1997.

\bibitem{sela13}
Z.~Sela.
\newblock Diophantine geometry over groups {VIII}: {S}tability.
\newblock {\em Ann. of Math.}, 177:787--868, 2013.

\bibitem{MR2542213}
N.~W.~M. Touikan.
\newblock The equation {$w(x,y)=u$} over free groups: an algebraic approach.
\newblock {\em J. Group Theory}, 12(4):611--634, 2009.

\bibitem{Wise2012}
D.~Wise.
\newblock {\em From Riches to {R}aags: 3-Manifolds, Right-Angled {A}rtin
  Groups, and Cubical Geometry}.
\newblock American Mathematical Society, 2012.

\end{thebibliography}

\end{document}